\documentclass{amsart}

\usepackage[utf8]{inputenc}
\usepackage{amssymb}
\usepackage[leqno]{amsmath}
\usepackage{stmaryrd}
\usepackage{tikz}
\usepackage{subfig}
\usepackage{todonotes}

\theoremstyle{plain}
\newtheorem{theorem}{Theorem}

\newtheorem{lemma}[theorem]{Lemma}
\newtheorem{proposition}[theorem]{Proposition}

\theoremstyle{remark}
\newtheorem{remark}[theorem]{Remark}

\newcommand{\N}{\mathbb{N}}
\newcommand{\R}{\mathbb{R}}
\newcommand{\Cleq}{\ensuremath{\lesssim}}  
\DeclareMathOperator{\diam}{diam}                             
\newcommand{\Mesh}{\mathcal{M}}
\newcommand{\Faces}[1]{\mathcal{F}_{#1}}
\newcommand{\FacesM}{\Faces{}}
\newcommand{\FacesMint}{\FacesM^i}
\newcommand{\Skel}{\Sigma}
\newcommand{\Jump}[1]{\left \llbracket #1 \right \rrbracket}%_{#2}}
\newcommand{\Avg}[1]{\left \{\!\!\left \{ #1 \right \}  \!\!\right \}}
\newcommand{\Shape}{\gamma}
\newcommand{\Normal}{n}
\newcommand{\RefSim}{K_\Ref}
\newcommand{\Ref}{{\mathrm{ref}}}
\newcommand{\degree}{\ell}
\newcommand{\Poly}[1]{\mathbb{P}_{#1}}                        
\newcommand{\Polyell}[1]{S_{#1}^0}           
\newcommand{\PolyellAvg}[1]{\widehat{S}_{#1}^0}
\newcommand{\Lagr}[1]{S_{#1}^1}              
\newcommand{\CR}[1]{\mathit{CR}_{#1}} 
\newcommand{\Domain}{\Omega}
\newcommand{\Dim}{d}
\newcommand{\Leb}[1]{L^2(#1)}                                
\newcommand{\LebH}[1]{L^2_0(#1)}
\newcommand{\Sob}[1]{H^1(#1)}                              
\newcommand{\SobH}[1]{H^1_0(#1)}
\newcommand{\SobHD}[1]{H^{-1}(#1)}
\DeclareMathOperator{\Div}{div}
\DeclareMathOperator{\DivM}{\Div_{\Mesh}}
\newcommand{\Divdisc}[1]{\Div_{#1}}
\newcommand{\Grad}{\nabla}
\DeclareMathOperator{\GradM}{\nabla_{\Mesh}}             
\DeclareMathOperator{\Rot}{rot}
\DeclareMathOperator{\Curl}{curl}
\newcommand{\Smt}{E}
\newcommand{\PiolaCon}{{\mathcal{P}^{con}_K}}
\newcommand{\PiolaCov}{{\mathcal{P}^{cov}_K}}
\newcommand{\snorm}[1]{\left \lvert #1 \right \rvert}          
\newcommand{\norm}[1]{\| #1 \|}                     
\newcommand{\normLeb}[2]{\norm{#1}_{\Leb{#2}}}
\newcommand{\opnorm}[3]{\norm{#1}_{\ifx#2#3\mathcal{L}(#2)\else\mathcal{L}(#2,#3)\fi}}       
\newcommand{\Visc}{\mu}
\newcommand{\DG}{\mathrm{dG}}
\begin{document}

% \title[Quasi-optimal and pressure robust dG discretizations of Stokes]{Quasi-optimal and pressure robust discontinuous Galerkin discretizations of the Stokes equations} 
% \author{\today}

\title[Quasi-optimal and pressure robust dG discretizations of
Stokes]{Quasi-optimal and pressure robust\\ discretizations of the
  Stokes equations by\\moment- and divergence-preserving operators}

\author[C.~Kreuzer]{Christian Kreuzer}
\address{TU Dortmund \\ Fakult{\"a}t f{\"u}r Mathematik \\ D-44221 Dortmund \\ Germany}
\email{christian.kreuzer@tu-dortmund.de}
\author[R.~Verf\"urth]{R\"udiger Verf\"urth}
\address{Ruhr-Universit\"at Bochum\\ Fakult{\"a}t f{\"u}r Mathematik
  \\ D-44780 Bochum \\ Germany}
\email{ruediger.verfuerth@rub.de}
\author[P.~Zanotti]{Pietro Zanotti}
\address{Universit\`a degli Studi di Milano \\ Dipartimento di Matematica 'F. Enriques' \\ I-20133 Milano \\ Italy}
\email{pietro.zanotti@unimi.it}



\begin{abstract}
We approximate the solution of the Stokes equations by a new
quasi-optimal and pressure robust
discontinuous Galerkin discretization of arbitrary order.
This means quasi-optimality of the velocity error independent of the
pressure. Moreover, the discretization is well-defined for any load
which is admissible for the continuous problem and it also provides
classical quasi-optimal estimates for the sum of velocity and pressure
errors. The key design principle is a careful discretization of the
load involving a linear operator, which maps discontinuous
Galerkin test functions onto conforming ones thereby preserving the
discrete divergence and certain moment conditions on faces and elements.
\end{abstract}

\maketitle

\section{Introduction}
\label{S:introduction}

This paper is a new contribution to the research programme initiated
in \cite{Kreuzer.Zanotti:19,Verfuerth.Zanotti:19}, which aims
at designing quasi-optimal and pressure robust discretizations of the
Stokes equations 
\begin{equation}
\label{Stokes-strong}
-\Visc \Delta u + \Grad p = f
\quad \text{and} \quad 
\Div u = 0
\quad \text{in } \Domain,
\qquad
u = 0 
\quad \text{on } \partial \Domain 
\end{equation}
for the largest possible class of inf-sup stable pairs of finite
element spaces. 

To illustrate our results, let $V/Q$ be an inf-sup stable pair and
assume that a given discretization produces an approximation
$(\overline{u}, \overline p) \in V \times Q$ to the solution $(u, p)$
of \eqref{Stokes-strong}. Moreover, let $\norm{\cdot}_1$ be a
$H^1$-like norm. We say that the given discretization is
\textit{quasi-optimal} when there is a constant $C_\mathrm{qo}\geq 1$
such that 
\begin{equation}
\label{qo}
\Visc \norm{u-\overline{u}}_1
+
\normLeb{p-\overline{p}}{\Domain}
\leq C_\mathrm{qo}
\left(\Visc  \inf_{v \in V} \norm{u-v}_1  + 
\inf_{q \in Q}\normLeb{p-q}{\Domain} \right). 
\end{equation} 
Analogously, we say that the given discretization is \textit{quasi-optimal and pressure robust} when there is a constant $C_\mathrm{qopr} \geq 1$ such that
\begin{equation}
\label{qopr}
\norm{u-\overline{u}}_1 
\leq C_{\mathrm{qopr}}
\inf_{v \in V} \norm{u-v}_1.
\end{equation}

Any discretization fulfilling the above error estimates 
\begin{itemize}
\item is defined for any admissible load $f$ in the weak formulation of \eqref{Stokes-strong}
\item inherits the approximation properties of the underlying spaces $V$ and $Q$, irrespective of the regularity of $(u, p)$ and $f$,
\item is pressure robust, in the sense that~\eqref{qopr} implies that large irrotational forces (or, equivalently, large pressure errors) do not affect the velocity error, cf. Remark~\ref{R:pressure-robustness} below.
\end{itemize}
Whereas the first two properties are desirable in the discretization of any equation, the third one is specific to the (Navier-)Stokes equations. Its importance has been pointed out in \cite{Linke:14} and further investigated in various other references, see e.g. \cite{John.Linke.Merdon.Neilan.Rebholz:17}. 

Most Stokes discretizations based on nonconforming pairs fail to
fulfill \eqref{qo}. Analogously, most discretizations with other pairs
than divergence-free ones fail to fulfill \eqref{qopr}. Both claims
follow from the abstract results in
\cite{Kreuzer.Zanotti:19,Veeser.Zanotti:18}. Indeed, the combination
of \eqref{qo} and \eqref{qopr} has been available for a long time only
for discretizations based on conforming and divergence-free pairs,
like the one of Scott and Vogelius~\cite{Scott.Vogelius:85}.
The importance of pressure robustness was observed in~\cite{Linke:14},
where pressure robust schemes are proposed using
\(H(\Div)\)-conforming maps applied to the test functions. As a trade
off, the quasi-optimality was weakened by involving additional
consistency errors; compare also with the overview article~\cite{John.Linke.Merdon.Neilan.Rebholz:17}. Here and in~\cite{Kreuzer.Zanotti:19,Verfuerth.Zanotti:19}, we design
quasi-optimal and pressure robust discretizations by devising, in particular, alternative $H^1$-conforming maps applied to test functions. 
 
The discretization proposed in \cite{Verfuerth.Zanotti:19} uses the first-order nonconforming Crouzeix-Raviart pair and can be written as follows: find $\overline u \in V$ and $\overline p \in Q$ such that 
\begin{equation}
\label{Stokes-disc}
\begin{alignedat}{2}
&\forall v \in V
&\qquad
\Visc \,a(\overline{u}, v)
+ b(v, \overline p) 
&= \left\langle  f , \Smt v \right\rangle  \\
&\forall q \in Q
&\qquad
b(\overline u, q) &= 0 
\end{alignedat}
\end{equation}  
where the forms $a$ and $b$ are as in the original discretization
described in \cite{Crouzeix.Raviart:73}. The operator $\Smt$ maps into
continuous piecewise polynomials and preserves the discrete divergence
and the averages on the mesh faces. This idea has been generalized in
\cite{Kreuzer.Zanotti:19} to a wide class of pairs, under the same
conditions on $\Smt$. The only difference is that the form $a$ needs
to be augmented with additional terms. In this paper we propose a
different approach, which does not require any augmentation of $a$, at
the price of a more involved construction of $\Smt$.  

More precisely, we propose a class of discontinuous Galerkin
discretizations of arbitrary order $\degree \geq 1$, which differ from
the ones in \cite{Hansbo.Larson:02} only in the use of an operator
$\Smt$ as in \eqref{Stokes-disc}. Here $\Smt$ is required to preserve
the discrete divergence and all moments up to the order $\degree-1$ on
the mesh faces and up to the order $\degree-2$ in the mesh
elements. The same approach applies also to
$H_\mathrm{div}$-conforming pairs \cite{Cockburn.Kanschat.Schotzau:07}
and with higher-order Crouzeix-Raviart pairs
\cite{Baran.Stoyan:07,Crouzeix.Falk:89}, but fails when dealing with
pairs involving a reduced integration of the divergence.  

The remaining part of this paper is organized as follows. In
section~\ref{S:abstract-theory} we propose the new discretization and
motivate the above-mentioned conditions on
$\Smt$. Section~\ref{S:smoother} is devoted to the construction of
$\Smt$ and to the derivation of the error estimates. Finally, in
section~\ref{S:numerical-experiments} we investigate numerically the
proposed discretization in the lowest-order case. We indicate Lebesgue
and Sobolev spaces and their norms as usual, see
e.g. \cite{Brenner.Scott:08}.  

\section{Discontinuous Galerkin discretization of the Stokes equations}
\label{S:abstract-theory}

\subsection{Stokes equations}
\label{SS:Stokes}

Let $\Domain \subseteq \R^\Dim$, $\Dim \in \{2,3\}$, be an open and
bounded polyhedron with Lipschitz boundary. The variational
formulation of the Stokes equations in
$\Domain$, with viscosity $ \Visc > 0$, load $f \in \SobHD{\Domain} :=
(\SobH{\Domain}^\Dim)'$ and homogeneous essential boundary conditions,
reads as follows: find a velocity $u \in \SobH{\Domain}^\Dim$ and a
pressure $p \in \LebH{\Domain}$ such that 
\begin{equation}
\label{Stokes}
\begin{alignedat}{2}
&\forall v \in \SobH{\Domain}^\Dim&
\qquad
\Visc \int_\Domain \Grad u \colon \Grad v 
- \int_\Domain p \Div v 
&= \left\langle  f , v \right\rangle  \\
&\forall q \in \LebH{\Domain}&
\qquad
\int_\Domain q \Div u &= 0.
\end{alignedat}
\end{equation}
Here $\colon$ denotes the euclidean scalar product of $\Dim \times
\Dim$ tensors and $\left\langle \cdot, \cdot \right\rangle $ is the
dual pairing of $\SobHD{\Domain}$ and $\SobH{\Domain}^\Dim$. Note that
we look for the pressure $p$ in the space $\LebH{\Domain} := \{ q \in
\Leb{\Domain} \mid \int_\Domain q = 0 
\}$, according to the boundary condition $u = 0$ on $\partial \Domain$. This problem is well-posed and we have
\begin{equation}
\label{stability-velocity+pressure}
\Visc \normLeb{\Grad u}{\Domain}
+
\normLeb{p}{\Domain}
\leq
c \norm{f}_{\SobHD{\Domain}}
\end{equation} 
where $c$ only depends on the geometry of $\Domain$, see, e.g., \cite[Theorem~8.2.1]{Boffi:Brezzi:Fortin.13}. Moreover, introducing the kernel of the divergence operator
\begin{equation*}
\label{kernel}
Z := \{ z \in \SobH{\Domain}^\Dim \mid \Div z = 0 \},
\end{equation*} 
we infer $u \in Z$ and the a priori estimate 
\begin{equation}
\label{stability-velocity}
\Visc \normLeb{\Grad u}{\Domain}
\leq
\norm{f_{|Z}}_{Z'}
:=
\sup_{z \in Z} \dfrac{\left\langle f , z \right\rangle }{\normLeb{\Grad z}{\Domain}}.
\end{equation}

\subsection{Meshes and polynomials}
\label{SS:meshes-spaces}

% Mesh
Let $\Mesh$ be a face-to-face simplicial mesh of $\Domain$. The shape constant $\Shape_\Mesh$ of $\Mesh$ is given by  
\begin{equation*}
\label{shape-constant}
\gamma_\Mesh := 
\max_{K \in \Mesh} \; \dfrac{h_K}{\rho_K}
\end{equation*}
where $h_K$ is the diameter of a $\Dim$-simplex $K \in \Mesh$ and $\rho_K$ is the diameter of the largest ball inscribed in $K$. We denote by $\FacesM$ and $\FacesMint$ the sets collecting all faces and all interior faces of $\Mesh$, respectively. The skeleton of $\Mesh$ is $\Skel := \cup_{F \in \FacesM} F$. We let the meshsize $h$ and the normal $\Normal$ be the piecewise constant functions on $\Sigma$ given by
\begin{equation*}
\label{meshsize+normal}
h_{|F} := \diam(F) 
\qquad \text{and} \qquad
\Normal_{|F} := n_F 
\end{equation*}
for all $F \in \FacesM$. Here $\Normal_F$ is a unit normal vector of $F$, pointing outside $\Domain$ if $F \subseteq \partial \Domain$.

% Operators for piecewise smooth functions
We denote by $\mathcal{D}_\Mesh$ the broken version of a differential operator $\mathcal{D}$, that is
\begin{equation*}
\label{broken-operators}
(\mathcal{D}_\Mesh v)_{|K} := \mathcal{D}(v_{|K})
\end{equation*}
for all $K \in \Mesh$ and for piecewise smooth $v$. % We introduce the
% jump and the average of $v$ respectively \(\GradM v\) on the skeleton
% of $\Mesh$ by setting for an interior face $F \in \FacesM$, $F \subseteq \Domain$
% \begin{align}{}\label{Jump+Avg-interior}
%   \begin{alignedat}{2}
%     \Jump{v}_{|F}&:=v_{|K_1} \Normal_{K_1} + v_{|K_2} \Normal_{K_2}&\qquad \Jump{\nabla
%       v}_{|F}&:=\GradM v_{|K_1} \cdot\Normal_{K_1} + \GradM v_{|K_2}\cdot \Normal_{K_2}
%     \\
%     \Avg{v}_{|F}&:=v_{|K_1} + v_{|K_2}\qquad \qquad&\Avg{\Grad
%       v}_{|F}&:=\GradM v_{|K_1} + \GradM v_{|K_2},
%   \end{alignedat}
% \end{align}
% where $K_1, K_2 \in \Mesh$ are such that $F = K_1 \cap K_2$ and
% $\Normal_{K_i}$ is the normal pointing outside $K_i$, \(i=1,2\). For boundary faces $F \in \FacesM$, $F
% \subseteq \partial \Domain$, it holds 
% \begin{equation*}
% \Jump{v}_{|F} 
% =
% v_{|K} \Normal_K,\quad
% \Avg{v}_{|F}=v_{|F}\quad\text{and}\quad \Jump{\nabla v}_{|F} (x) =
% \GradM v_{|F}\cdot\normal_K,\quad
% \Avg{\Grad v}_{|F}=\GradM v_{|F}.
% \end{equation*}
% where $K \in \Mesh$ is such that $F = K \cap \partial \Domain$. To
% alleviate the notation, we wrote $\Jump{\Grad \cdot}$ and $\Avg{\Grad
%   \cdot}$ in place of $\Jump{\GradM \cdot}$ and $\Avg{\GradM \cdot}$. 
We indicate by \(\Jump{v}\) and \(\Avg{v}\), respectively, the jump and
the average of \(v\) on the skeleton $\Skel$ of \(\Mesh\). 
More precisely, for an interior face $F \in \FacesMint$ and  for $x \in F$, we have 
\begin{equation*}
\label{Jump+Avg-interior}
\Jump{v}_{|F}(x)
=
v_{|K_1}(x) - v_{|K_2} (x)
\quad \text{ and } \quad
\Avg{v}_{|F}(x) 
= 
\dfrac{v_{|K_1}(x) + v_{|K_2} (x)}{2} 
\end{equation*} 
where $K_1, K_2 \in \Mesh$ are such that $F = K_1 \cap K_2$ and
$\Normal$ points outside $K_1$. Note that the sign of \(\Jump{v}\)
depends on the orientation of $\Normal$, which will however not be
significant to our discussion. For boundary faces $F \in \FacesM \setminus \FacesMint$, it holds 
\begin{equation*}
\Jump{v}_{|F} (x) 
=
v_{|K} (x)
=
\Avg{v}_{|F} (x),
\end{equation*}
where $K \in \Mesh$ is such that $F = K \cap \partial \Domain$. To
alleviate the notation, we write $\Jump{\Grad \cdot}$ and $\Avg{\Grad
  \cdot}$ in place of $\Jump{\GradM \cdot}$ and $\Avg{\GradM \cdot}$.   

% Piecewise polynomials
The spaces $\Poly{\degree}(K)$ and $\Poly{\degree}(F)$, $\degree \geq
0$, consist of all polynomials of total degree $\leq \degree$ on a
$\Dim$-simplex $K \in \Mesh$ and a face $F \in \FacesM$,
respectively. For convenience, we set $\Poly{-1} = \{ 0 \}$. The space
of broken polynomials on $\Mesh$ with total degree $\leq \degree$
reads 
\begin{equation*}
\Polyell{\degree} 
:=
\{ v : \Domain \to \R  \mid \forall K \in \Mesh \;\; {v}_{|K} \in \Poly{\degree}(K) \}.
\end{equation*}
The approximation of the pressure space involved in the Stokes equations \eqref{Stokes} motivates the use of the one-codimensional subspace
\begin{equation*}
\PolyellAvg{\degree}
:=
\Polyell{\degree} \cap \LebH{\Domain}. 
\end{equation*}

We shall repeatedly make use of the following integration by parts formula
\begin{equation}
\label{ibp-elementwise}
\int_\Domain (\DivM v) q 
=
-\int_\Domain v \cdot \GradM q 
+ \int_\Skel \Jump{v} \cdot \Normal \Avg{q}
+\int_{\Skel \setminus \partial \Domain} \Avg{v} \cdot \Normal \Jump{q}
\end{equation}
where $v \in (\SobH{\Domain} +\Polyell{\degree})^\Dim$ and $q \in \Polyell{\degree-1}$, see e.g. \cite[equation~(3.6)]{Arnold.Brezzi.Cockburn.Marini:02}.

\subsection{Discontinuous Galerkin discretization}
\label{SS:dG-discretization}

% Bilinear forms
We consider a discontinuous Ga\-lerkin (dG) discretization of order
$\degree \in \N$ of the Stokes equations (see, for instance,
\cite{Hansbo.Larson:02}) that builds on the bilinear forms $a_\DG:
(\Polyell{\degree})^\Dim \times (\Polyell{\degree})^\Dim \to \R$ and
$b_\DG: (\Polyell{\degree})^\Dim \times \PolyellAvg{\degree-1} \to \R$
given by 
\begin{equation}
\label{DG-form-a}
\begin{split}
a_\DG(w, v) := & 
\int_{\Domain} \GradM w \colon \GradM v - 
\int_\Skel \Avg{\Grad w} \Normal \cdot \Jump{v}\\
&-\int_\Skel \Jump{w} \cdot \Avg{\Grad v} \Normal + 
\int_{\Skel} \dfrac{\eta}{h} \Jump{w} \cdot \Jump{v}
\end{split}
\end{equation} 
and
\begin{equation}
\label{DG-form-b}
b_\DG(w, q) :=
-\int_\Domain q \DivM w  + 
\int_\Skel \Jump{w} \cdot \Normal \Avg{q}
\end{equation}
where $\eta > 0$ is a penalty parameter. 
%

% Smoothing operator
Motivated by the abstract results in \cite{Veeser.Zanotti:18}, we let
$\Smt_\DG: (\Polyell{\degree})^\Dim \to \SobH{\Domain}^\Dim$ be a
linear operator and 
% Discontinuous Galerkin discretization
consider the following dG discretization of the Stokes equations \eqref{Stokes}: find a discrete velocity $u_\DG \in (\Polyell{\degree})^\Dim$ and a discrete pressure $p_\DG \in \PolyellAvg{\degree-1}$ such that
\begin{equation}
\label{Stokes-dG}
\begin{alignedat}{2}
&\forall v \in (\Polyell{\degree})^\Dim&
\qquad
\Visc \, a_\DG(u_\DG, v) + b_\DG(v, p_\DG) 
&= \left\langle  f , \Smt_\DG v \right\rangle  \\
&\forall q \in \PolyellAvg{\degree-1}&
\qquad
b_\DG(u_\DG, q) &= 0.
\end{alignedat}
\end{equation}

% Discrete divergence and equivalent formulation
Introducing the discrete divergence $\Divdisc{\DG}: (\Polyell{\degree})^\Dim \to \PolyellAvg{\degree-1}$ through the problem
\begin{equation}
\label{discrete-divergence}
\forall q \in \PolyellAvg{\degree-1}
\qquad
\int_\Domain q \Divdisc{\DG} w
=
-b_\DG(w, q)
\end{equation}
for all $w\ \in (\Polyell{\degree})^\Dim$, we can equivalently rewrite \eqref{Stokes-dG} as follows
\begin{equation*}
\label{Stokes-dG-equiv}
\begin{alignedat}{2}
&\forall v \in (\Polyell{\degree})^\Dim&
\qquad
\Visc \,a_\DG(u_\DG, v) - \int_\Domain  p_\DG \Divdisc{\DG} v
&= \left\langle  f , \Smt_\DG v \right\rangle  \\
&\forall q \in \PolyellAvg{\degree-1}&
\qquad
\int_\Domain q \Divdisc{\DG} u_\DG &= 0.
\end{alignedat}
\end{equation*}
This shows that 
\begin{equation}
\label{kernel-DG}
u_\DG \in Z_\DG := 
\{ w \in (\Polyell{\degree})^\Dim \mid \Divdisc{\DG} w = 0 \}
\end{equation}
i.e. the discrete velocity $u_\DG$ belongs to the kernel of the discrete divergence. 

% Alternative definition of discrete divergence
\begin{remark}[Alternative definition of $\Divdisc{\DG}$]
\label{R:discrete-divergence}
Note that we could equivalently define the discrete divergence as the linear operator $\Divdisc{\DG}: (\Polyell{\degree})^\Dim \to \Polyell{\degree-1}$ given by 
\begin{equation*}
\label{discrete-divergence-alternative}
\forall q \in \Polyell{\degree-1}
\qquad
\int_\Domain q \Divdisc{\DG} w
=
\int_\Domain q \DivM w  - 
\int_\Skel \Jump{w} \cdot \Normal \Avg{q}.
\end{equation*}
Indeed, testing with $q = 1$ and integrating by parts as in \eqref{ibp-elementwise}, we see that
\begin{equation*}
\int_\Domain \Divdisc{\DG} w
= 
\int_\Domain \DivM w 
- \int_\Skel \Jump{w} \cdot \Normal 
= 0
\end{equation*}
for all $ w \in (\Polyell{\degree})^\Dim$. This proves that $\Divdisc{\DG} w \in \PolyellAvg{\degree-1}$. Then, testing with $q \in \PolyellAvg{\degree}$, we retrieve \eqref{discrete-divergence}. 
\end{remark}

% Error notions
To assess the quality of the discretization \eqref{Stokes-dG}, we introduce the scalar product
\begin{equation*}
(w, v)_\DG
:= 
\int_\Domain \GradM w \colon \GradM v
+
\int_ 
\Skel \dfrac{\eta}{h} \Jump{w} \cdot \Jump{v},
\qquad
w, v \in \SobH{\Domain}^\Dim + (\Polyell{\degree})^\Dim
\end{equation*}
where the penalty parameter $\eta$ is the same as in \eqref{DG-form-a}. We measure the velocity error $u-u_\DG$ in the norm $\norm{\cdot}_\DG$ induced by $(\cdot, \cdot)_\DG$, that is an extension of the norm $\normLeb{\Grad \cdot}{\Domain}$ to $(\SobH{\Domain} + \Polyell{\degree})^\Dim$. Since $\PolyellAvg{\degree-1} \subseteq \Leb{\Domain}$, we measure the pressure error $p - p_\DG$ in the $L^2$-norm.

% Notation: label dG
\begin{remark}[Notation for $\DG$ discretization]
\label{R:notation}
The label `$\DG$' identifies all objects and quantities that specifically depend on the discretization \eqref{Stokes-dG}. In most (but not all) cases, such objects and quantities depend on the penalty parameter $\eta$.
\end{remark}

% Notation: hidden constant
In what follows, we write $C$ for a positive nondecreasing function of the shape constant $\Shape_\Mesh$ of $\Mesh$. Such function may depend also on other parameters (like $\Domain$, $\Dim$, or $\degree$) but is independent of the viscosity $\Visc$ and the penalty parameter $\eta$. Furthermore, the value of $C$ does not need to be the same at different occurrences. We sometimes abbreviate $A \leq C B$ as $A \Cleq B$. %and $C^{-1} B \leq A \leq C B$ as $A \eqsim B$.

\subsection{Stability}
\label{SS:stability}

% Coercivity and boundedness
The so-called inverse trace inequality \cite[Lemma~1.46]{DiPietro.Ern:12} implies that there is a
constant $\overline{\eta} > 0$, depending only on the shape parameter
of $\Mesh$ and the polynomial degree \(\degree\), such that 
\begin{equation}
\label{DG-threshold}
\int_{\Skel} h \snorm{\Avg{v}}^2
\leq \overline{\eta} \normLeb{v}{\Domain}^2\qquad\forall v \in  (\Polyell{\degree})^{\Dim \times \Dim}.
\end{equation} 
Hence, simple algebraic manipulations reveal that the form $a_\DG$ is bounded and coercive. More precisely, we have
\begin{subequations}
\label{coercivity+boundedness+infsup-DG}
\begin{equation}
\label{boundedness-DG}
a_\DG(w,v) \leq 
\overline{\alpha}_\DG\norm{w}_\DG \norm{v}_\DG,
\qquad
\overline{\alpha}_\DG := 1+\sqrt{\overline{\eta}/\eta}.
\end{equation}
and 
\begin{equation}
\label{coercivity-DG}
a_\DG(w,w) \geq
\underline{\alpha}_\DG
\norm{w}_\DG^2, \qquad
\underline{\alpha}_\DG:= 1 - \sqrt{\overline{\eta} / \eta}
\end{equation}
for all $w, v \in (\Polyell{\degree})^\Dim$.
%
% Inf-sup stability
Furthermore, the form $b_\DG$ is inf-sup stable, in that 
\begin{equation}
\label{inf-sup-DG}
\beta_\DG  \normLeb{q}{\Domain}
\leq 
\sup_{w \in (\Polyell{\degree})^\Dim}
\dfrac{b_\DG(w, q)}{\norm{w}_\DG},
\qquad \beta_\DG^{-1} \Cleq \max\{1, \sqrt{\eta}\}
\end{equation}	
for all $q \in \PolyellAvg{\degree-1}$, see e.g. \cite[section~4.4]{John.Linke.Merdon.Neilan.Rebholz:17}. Note that, without loss of generality, we can assume $\beta_\DG \leq 1$. 
\end{subequations}

% Discrete well-posedness and stability
The following discrete counterpart of \eqref{stability-velocity+pressure} follows from \eqref{coercivity+boundedness+infsup-DG} and the theory of saddle point problems. The discrete stability constant involves, in particular, the operator norm of $\Smt_\DG$
\begin{equation*}
\norm{\Smt_\DG} := \opnorm{\Smt_\DG}{(\Polyell{\degree})^\Dim}{\SobH{\Domain}^\Dim}.
\end{equation*}
\begin{lemma}[Discrete well-posedness and stability]
\label{L:stability-velocity+pressure-DG}
Let $\overline{\eta} > 0$ be as in \eqref{DG-threshold} and assume
$\eta > \overline{\eta}$. The discretization \eqref{Stokes-dG}, with
viscosity $\Visc > 0$ and load $f \in \SobHD{\Domain}$, is uniquely
solvable and its solution $(u_\DG, p_\DG)$ satisfies the a priori
estimate
\begin{equation*}
\Visc \norm{u_\DG}_\DG 
\leq 
\dfrac{1}{\underline{\alpha}_\DG} \norm{\Smt_\DG}
 \norm{f}_{\SobHD{\Domain}}
\quad \text{and} \quad
\normLeb{p_\DG}{\Domain}
\leq 
\dfrac{2\overline{\alpha}_\DG}{\underline{\alpha}_\DG\beta_\DG} \norm{\Smt_\DG}
 \norm{f}_{\SobHD{\Domain}}.
\end{equation*}
% \begin{equation*}
% % \label{stability-velocity+pressure-DG}
% \Visc \norm{u_\DG}_\DG + \normLeb{p_\DG}{\Domain}
% \leq
% \dfrac{1}{\underline{\alpha}_\DG}
% \left( 1+ \dfrac{2\overline{\alpha}_\DG}{\beta_\DG}\right) 
% \norm{\Smt_\DG}
% \norm{f}_{\SobHD{\Domain}}.
% \end{equation*}
\end{lemma} 

\begin{proof}
Since $(\Polyell{\degree})^\Dim$ is finite dimensional, the operator $\Smt_\DG$ is bounded. This implies that the adjoint operator $\Smt_\DG^\star$ is well-defined and that the load in the first equation of \eqref{Stokes-dG} is $\Smt_\DG^\star f$. Then \cite[Theorem~4.2.3]{Boffi:Brezzi:Fortin.13} implies that \eqref{Stokes-dG} is uniquely solvable, as a consequence of \eqref{coercivity+boundedness+infsup-DG}, and yields the a priori estimates
\begin{equation*}
\Visc \norm{u_\DG}_\DG 
\leq 
\dfrac{1}{\underline{\alpha}_\DG} \norm{\Smt_\DG^\star f}
\qquad \text{and} \qquad
\normLeb{p_\DG}{\Domain}
\leq 
\dfrac{2\overline{\alpha}_\DG}{\underline{\alpha}_\DG\beta_\DG} 
\norm{\Smt_\DG^\star f}
\end{equation*}
where $\norm{\Smt_\DG^\star f}$ is the norm of the functional
$\Smt_\DG^\star f$ in the dual space of $(\Polyell{\degree})^\Dim$. We
conclude by recalling that the operator norm $\Smt_\DG^\star$ coincides
with the one of $\Smt_\DG$, see \cite[Remark~2.16]{Brezis:11}. 
\end{proof}

% Stability of the discrete velocity
A discrete counterpart of \eqref{stability-velocity} additionally holds, under the assumption that $\Smt_\DG$ maps discretely divergence-free functions into exactly divergence-free functions. To our best knowledge, the importance of this condition was first pointed out in \cite{Linke:14}. 
\begin{lemma}[Stability of the discrete velocity]
\label{L:stability-velocity-DG}
Under the assumptions of Lemma~\ref{L:stability-velocity+pressure-DG},
for any load \(f\in \SobHD{\Domain} \)
the discrete velocity $u_\DG \in (\Polyell{\degree})^\Dim$ in
\eqref{Stokes-dG} additionally enjoys the a priori estimate 
\begin{equation}
\label{stability-velocity-DG}
\Visc \norm{u_\DG}_\DG \leq 
\dfrac{\norm{\Smt_\DG}}{\underline{\alpha}_\DG} \norm{f_{|Z}}_{Z'}
\end{equation} 
if and only if 
\begin{equation}
\label{kernel-inclusion}
\Smt_\DG(Z_\DG) \subseteq Z.
\end{equation}
\end{lemma}

\begin{proof}
  %<=
Assume first that \eqref{kernel-inclusion} holds. Testing the first equation of \eqref{Stokes-dG} with the elements of $Z_\DG$, we see that $u_\DG$ solves the reduced problem
\begin{equation*}
\forall v \in Z_\DG \qquad
\Visc \, a_\DG(u_\DG, v) = \left\langle f, \Smt_\DG v\right\rangle. 
\end{equation*}	
In view of the inclusion \eqref{kernel-DG}, we are allowed to set $v = u_\DG$ and exploit the coercivity \eqref{coercivity-DG} of $a_\DG$
\begin{equation*}
\Visc \underline{\alpha}_\DG\norm{u_\DG}_\DG^2
\leq
\left\langle f, \Smt_\DG u_\DG\right\rangle.  
\end{equation*} 
Then, the inclusion $\Smt_\DG u_\DG \in Z$ implies
\begin{equation*}
\left\langle f, \Smt_\DG u_\DG\right\rangle
\leq
\norm{f_{|Z}}_{Z'} \norm{\Smt_\DG u_\DG}. 
\end{equation*}
We derive the claimed a priori estimate in view of the boundedness of
$\Smt_\DG$.

%=>
Conversely, assume \eqref{stability-velocity-DG} holds and
there is $v \in Z_\DG$ such that $\Div \Smt_\DG v \neq 0$. Set $f :=
\Grad(\Div \Smt_\DG v) \in \SobHD{\Domain}$. On the one hand, we have
$f_{|Z} = 0$, so that \eqref{stability-velocity-DG} implies $u_\DG =
0$. On the other hand, the boundedness of $a_\DG$ and the first equation of \eqref{Stokes-dG} reveal that $\Visc \overline{\alpha}_\DG \norm{u_\DG}_\DG \norm{v}_\DG \geq
\left\langle f, \Smt_\DG v\right\rangle = \normLeb{\Div \Smt_\DG
  v}{\Domain}^2 \neq 0$. This contradiction confirms that $\Smt_\DG$
maps $Z_\DG$ into $Z$ whenever \eqref{stability-velocity-DG} holds. 
\end{proof}

\begin{remark}[Pressure robustness]
\label{R:pressure-robustness}

The a priori estimate \eqref{stability-velocity} reveals that the velocity $u$ in the Stokes equations \eqref{Stokes} solely depends on $f_{|Z}$. In particular, this entails that $u$ is invariant with respect to irrotational perturbations of $f$, which only affect the pressure $p$, see Linke \cite{Linke:14}. Whenever the estimate \eqref{stability-velocity-DG} holds, the discretization \eqref{Stokes-dG} reproduces such invariant property and we call it `pressure robust'. We refer to
\cite{John.Linke.Merdon.Neilan.Rebholz:17} and to the references therein for an extensive discussion on the importance of pressure robustness in the discretization of the (Navier-)Stokes equations.  
\end{remark}

\subsection{Quasi-optimality}
\label{SS:quasi-optimality}

% Approximation by discretely divergence-free functions
We now look for conditions ensuring that the discretization \eqref{Stokes-dG} enjoys \eqref{qo}. To this end, we first investigate the approximation of the velocity field $u$ in \eqref{Stokes} by $Z_\DG$, i.e. by discretely divergence-free velocity fields. This is a standard question motivated by the inclusion \eqref{kernel-DG} and several related results are available in the literature. We refer to \cite[Theorem~12.5.17]{Brenner.Scott:08} for conforming discretizations and to \cite[Lemma~8.1]{Toselli:02} for dG discretizations.  

\begin{lemma}[Approximation by $Z_\DG$]
\label{L:discretely-divfree-approx}
Let $u \in \SobH{\Domain}^\Dim$ be the velocity solving \eqref{Stokes}. Then, there is a constant $\delta_\DG$ such that
\begin{equation*}
\label{discretely-divfree-approx}
\inf_{z \in Z_\DG} \norm{u-z}_\DG
\leq \delta_\DG
\inf_{w \in (\Polyell{\degree})^\Dim} \norm{u-w}_\DG.
\end{equation*}
Moreover, it holds $\delta_\DG \leq 1 + C \beta_\DG^{-1}$.
\end{lemma}
\begin{proof}
Let $w \in (\Polyell{\degree})^\Dim$ be given and denote by
$Z_\DG^\perp$ the orthogonal complement of $Z_\DG$ with respect to the
scalar product $(\cdot, \cdot)_\DG$. Inequality
\eqref{inf-sup-DG} implies that $Z_\DG^\perp$ and
$\PolyellAvg{\degree-1}$ have the same space dimension and  
\begin{equation*}
\beta_\DG \normLeb{q}{\Domain}
\leq 
\sup_{\widetilde{w} \in Z_\DG^\perp}
\dfrac{b_\DG(\widetilde{w}, q)}{\norm{\widetilde{w}}_\DG},
\end{equation*}
cf. \cite[Chapter~12.5]{Brenner.Scott:08}. Then, the Banach-Ne\u{c}as theorem (see, e.g., \cite[Theorem~2.6]{Ern.Guermond:04}) ensures the existence of a unique solution $\widetilde{w} \in Z_\DG^\perp$ to the problem 
\begin{equation*}
\forall q \in \PolyellAvg{\degree-1}
\qquad
b_\DG(\widetilde{w}, q) 
=
b_\DG(w, q)
\end{equation*}
together with the estimate 
\begin{equation*}
\beta_\DG\norm{\widetilde{w}}_\DG \leq 
\norm{b_\DG(w, \cdot)}_{(\PolyellAvg{\degree-1})'}.
\end{equation*} 
Recall that $u$ is in $\SobH{\Domain}$ and divergence-free, in view of the second equation of \eqref{Stokes}. Hence, for all $q \in \PolyellAvg{\degree-1}$, we have
\begin{equation*}
b_\DG(w, q) = \int_\Domain \DivM(u-w)q - \int_\Skel \Jump{u-w} \cdot \Normal \Avg{q} \Cleq \norm{u-w}_\DG \normLeb{q}{\Domain}
\end{equation*}
where we have used the inverse trace inequality \eqref{DG-threshold} for the term involving $\Avg{q}$. This estimate and the previous one entail that $\beta_\DG\norm{\widetilde{w}}_\DG \leq C \norm{u-w}_\DG$. Next, we set $z := w - \widetilde{w}$. By definition, we have $b_\DG(z, \cdot) = 0$, showing that $z \in Z_\DG$. Moreover, it holds
\begin{equation*}
\norm{u-z}_\DG
\leq
\norm{u-w}_\DG + \norm{\widetilde{w}}_\DG
\leq (1 + C\beta_\DG^{-1})
\norm{u-w}_\DG.
\end{equation*}
We conclude taking the infimum over all $w \in (\Polyell{\degree})^\Dim$. 
\end{proof}

% Alternative bound by Fortin operators
\begin{remark}[Size of $\delta_\DG$]
\label{L:size-C}
The bound of $\delta_\DG$ in the above lemma is known to be potentially pessimistic if $\beta_\DG$ is close to zero as, for instance, in channel-like stretched domains. A sharper bound of $\delta_\DG$ could be obtained in terms of the norm of a Fortin operator by arguing in the spirit of \cite[Remark~4.1]{John.Linke.Merdon.Neilan.Rebholz:17}.
\end{remark}

% Consistency for quasi-optimality
After this preparation, we observe that the dG discretization \eqref{Stokes-dG} fits into the abstract framework of \cite[section~2]{Kreuzer.Zanotti:19}. Hence, applying \cite[Lemma~2.6]{Kreuzer.Zanotti:19}, we derive that the following consistency conditions are necessary for quasi-optimality \eqref{qo}
\begin{subequations}
	\label{quasi-optimal-nec}
	\begin{alignat}{2}
	\label{quasi-optimal-nec-lapl}
	&\forall w \in Z \cap Z_\DG, \, v \in (\Polyell{\degree})^\Dim &\qquad&
	a_\DG(w, v) = \int_\Domain \Grad w \colon \Grad \Smt_\DG v\\
	%\intertext{and}
	\label{quasi-optimal-nec-div}
	&\forall w \in (\Polyell{\degree})^\Dim, \, q \in \PolyellAvg{\degree-1} &\qquad
	&b_\DG(w, q) = -\int_\Domain q \Div \Smt_\DG w.
	\end{alignat}
\end{subequations}
Differently from \cite{Kreuzer.Zanotti:19}, we deal with these
conditions assuming that $\Smt_\DG$ preserves sufficiently many
moments of $v$ on the $\Dim$-simplices and on the faces of $\Mesh$, in
the vein of \cite[section~3.2]{Veeser.Zanotti:18b}. 
\begin{lemma}[Consistency by moment-preserving operators]
\label{L:consistency-qopt}
Assume that the operator $\Smt_\DG: (\Polyell{\degree})^\Dim \to \SobH{\Domain}^\Dim$ is such that 
\begin{subequations}
	\label{consistency-qopt-conditions}
	\begin{align}
	\label{consistency-qopt-faces}
	\forall F \in \FacesMint, \; m_F \in \Poly{\degree-1}(F)^\Dim \quad
	&\int_F \Smt_\DG v \cdot m_F
	=
	\int_F \Avg{v} \cdot m_F\\
	%\end{align}
	%\begin{align}
	\label{consistency-qopt-simplices}
	\forall K \in \Mesh,\: m_K \in \Poly{\degree-2} (K)^\Dim
	\quad
	&\int_K \Smt_\DG v \cdot m_K
	=
	\int_K v \cdot m_K
	\end{align}
\end{subequations}
for all $v \in (\Polyell{\degree})^\Dim$. Then, $\Smt_\DG$ satisfies conditions \eqref{quasi-optimal-nec-lapl} and \eqref{quasi-optimal-nec-div}.
\end{lemma} 
\begin{proof}
Let $w \in (\Polyell{\degree})^\Dim$ and $q \in
\PolyellAvg{\degree-1}$. The integration by parts
formula \eqref{ibp-elementwise} yields 
\begin{equation}\label{eq:div_int_by_parts}
b_\DG(w, q)
=
\int_\Domain w \cdot \GradM q
- \int_{\Skel \setminus \partial \Domain} \Avg{w} \cdot \Normal \Jump{q}.
\end{equation}
In view of \eqref{consistency-qopt-conditions}, we can replace $w$ by $\Smt_\DG w$ in this identity. Then, we integrate back by parts and note that $\Jump{\Smt_\DG w} = 0$ on $\Skel
$, because of the inclusion $\Smt_\DG w \in \SobH{\Domain}^\Dim$,
\begin{equation*}
b_\DG(w, q)
=
\int_\Domain \Smt_\DG w \cdot \GradM q
- \int_{\Skel \setminus \partial \Domain} \Smt_\DG w \cdot \Normal \Jump{q}
=
-\int_\Domain q \Div \Smt_\DG w.
\end{equation*} 
This entails that \eqref{quasi-optimal-nec-div} holds. Arguing similarly, we infer that  
\begin{equation}
\label{consistency-qopt-velocity}
a_\DG(w,v) = \int_\Domain \GradM w \colon \Grad \Smt_\DG v
- \int_\Skel \Jump{w} \cdot \Avg{\Grad v} \Normal
+ \int_\Skel \dfrac{\eta}{h} \Jump{w} \cdot \Jump{v} 
\end{equation}
for all $w, v \in (\Polyell{\degree})^\Dim$, cf. \cite[Lemma~3.1]{Veeser.Zanotti:18b}. Hence, we conclude that \eqref{quasi-optimal-nec-lapl} holds, because $\Jump{w} = 0$ on $\Skel$ in view of the inclusion $w \in Z$. 
\end{proof}

\begin{remark}[Alternative approach to consistency]
\label{R:alternative-consistency}
The implication \eqref{consistency-qopt-conditions} $\implies$ \eqref{quasi-optimal-nec} stated in the previous lemma relies on our choice of the forms $a_\DG$ and $b_\DG$ and fails to hold for different discretizations. Roughly speaking, this happens whenever some sort of reduced integration of the divergence is involved. In such cases the consistency conditions \eqref{quasi-optimal-nec} need to be accommodated differently, for instance by the augmented Lagrangian formulation proposed in \cite{Kreuzer.Zanotti:19}.  
\end{remark}

% Quasi-optimality
The next theorem states that \eqref{consistency-qopt-conditions} is indeed a sufficient condition for quasi-optimality. The essence of this result and a partial proof can be found also in \cite[section~6]{Badia.Codina.Gudi.Guzman:14}.  
\begin{theorem}[Quasi-optimality]
\label{T:quasi-optimality}
Assume that the operator $\Smt_\DG$ satisfies \eqref{consistency-qopt-conditions} for all $v \in (\Polyell{\degree})^\Dim$. Moreover, let $\eta > \overline{\eta}$, where $\overline{\eta}$ is as in \eqref{DG-threshold}. Then, denoting by $(u,p)$ and $(u_\DG, p_\DG)$ the solutions of \eqref{Stokes} and \eqref{Stokes-dG}, respectively, with viscosity $\Visc > 0$ and load $f \in \SobHD{\Domain}$, we have
\begin{equation*}
\label{quasi-optimality-velocity}
\Visc \norm{u-u_\DG}_\DG
\Cleq 
\dfrac{1+\norm{E_\DG}}{\underline{\alpha}_\DG}
\left( 
\delta_\DG \Visc 
\inf_{w \in (\Polyell{\degree})^\Dim}
\norm{u-w}_{\DG}
+
\inf_{q \in \PolyellAvg{\degree-1}}
\normLeb{p-q}{\Domain}
\right) 
\end{equation*}
and
\begin{equation*}
\label{quasi-optimality-pressure}
\normLeb{p-p_\DG}{\Domain}
\Cleq
\dfrac{(1+\norm{\Smt_\DG})^2}{\underline{\alpha}_\DG\beta_\DG}
\left(  
\delta_\DG \Visc 
\inf_{w \in (\Polyell{\degree})^\Dim}
\norm{u-w}_{\DG}
+
\inf_{q \in \PolyellAvg{\degree-1}}
\normLeb{p-q}{\Domain}
\right) . 
\end{equation*}
\end{theorem}

\begin{proof}
%Velocity Estimate  
We first estimate the velocity error. For this purpose, let $\Pi u \in
Z_\DG$ be the $(\cdot, \cdot)_\DG$-orthogonal projection of $u$ onto
$Z_\DG$, that is 
\begin{equation*}
\label{DG-Ritz-projection}
\forall v \in Z_\DG \qquad
\int_\Domain \GradM(\Pi u - u) \colon \GradM v
+
\int_\Skel \dfrac{\eta}{h} \Jump{\Pi u} \cdot \Jump{v}
= 0.
\end{equation*} 
Setting $z := u_\DG - \Pi u$, the coercivity \eqref{coercivity-DG} and
the first equation of problems \eqref{Stokes} and \eqref{Stokes-dG}
yield 
\begin{equation}
\label{quasi-optimality-velocity-proof}
\begin{split}
\underline{\alpha}_\DG\norm{u_\DG  - \Pi u}_\DG^2
\leq
&\left( \int_\Domain \Grad u \colon \Grad \Smt_\DG z
- a_\DG(\Pi u, z)\right) +\\ 
&- \dfrac{1}{\Visc}\left( 
\int_\Domain p \Div \Smt_\DG z + b_\DG(z, p_\DG) \right)
=: \mathfrak{T}_1 - \dfrac{\mathfrak{T}_2}{\Visc} . 
\end{split}
\end{equation}
We bound $\mathfrak{T}_1$ according to the definition of $\Pi u$,
identity \eqref{consistency-qopt-velocity} (whose validity is guaranteed
by Lemma~\ref{L:consistency-qopt}) and inequality \eqref{DG-threshold} 
\begin{equation*}
\begin{split}
\mathfrak{T}_1
&=
\int_\Domain \GradM (u - \Pi u )\colon \GradM (\Smt_\DG z - z)
+ \int_\Sigma \Jump{\Pi u} \cdot \Avg{\Grad z} \cdot \Normal\\
&\leq (2+ \norm{\Smt_\DG}) \norm{u-\Pi u}_\DG \norm{z}_\DG.
\end{split}
\end{equation*}
We bound $\mathfrak{T}_2$ in view of the inclusion $z \in Z_\DG$
(which implies \(b_\DG(z, p_\DG)=0\)),
identity \eqref{quasi-optimal-nec-div} (whose validity is guaranteed
by Lemma~\ref{L:consistency-qopt}) and
\cite[Lemma~2.1]{Pyo.Nochetto:05}. Thus, we obtain 
\begin{equation*}
\mathfrak{T}_2 = 
\int_\Domain (p-q) \Div \Smt_\DG z
\leq 
\norm{\Smt_\DG} \norm{z}_\DG \normLeb{p-q}{\Domain}
\end{equation*} 
for all $q \in \PolyellAvg{\degree-1}$. We insert the estimates of $\mathfrak{T}_1$ and $\mathfrak{T}_2$ into \eqref{quasi-optimality-velocity-proof} and apply the triangle inequality, to obtain
\begin{equation*}
\norm{u-u_\DG}_\DG \leq
\dfrac{3 + \norm{\Smt_\DG}}{\underline{\alpha}_\DG}
\inf_{z \in Z_\DG} \norm{u-z}_\DG
+
\dfrac{\norm{\Smt_\DG}}{\Visc \underline{\alpha}_\DG} 
\inf_{q \in \PolyellAvg{\degree-1}}\normLeb{p-q}{\Domain},
\end{equation*}
where we have used \(\underline{\alpha}_\DG\le 1\).
We derive the claimed estimate of the velocity error by invoking
Lemma~\ref{L:discretely-divfree-approx}.

%Pressure estimate
Next, in order to estimate the pressure error, let $R p \in
\PolyellAvg{\degree-1}$ be the $L^2$-orthogonal projection of $p$ onto
$\PolyellAvg{\degree-1}$. The inf-sup stability \eqref{inf-sup-DG} and
the triangle inequality yield 
\begin{equation}
\label{quasi-optimality-velocity-proof2}
\normLeb{p-p_\DG}{\Domain}
\leq
\normLeb{p-Rp}{\Domain}
+ \dfrac{1}{\beta_\DG}
\sup_{v \in (\Polyell{\degree})^\Dim}
\dfrac{ b_\DG(v, p_\DG - Rp)}{\norm{v}_\DG}.
\end{equation}
Let $v \in (\Polyell{\degree})^\Dim$. The first equations of problems \eqref{Stokes} and \eqref{Stokes-dG} reveal 
\begin{equation}
\label{quasi-optimality-velocity-proof3}
\begin{split}
b_\DG(v, p_\DG - R p)
& =
\Visc \left( \int_\Domain \Grad u \colon \Grad \Smt_\DG v - a_\DG(u_\DG, v) \right) + \\ 
&\quad - \left( \int_\Domain p \Div \Smt_\DG v + b_\DG(v, Rp) \right) 
=: \Visc \mathfrak{U}_1 + \mathfrak{U}_2.
\end{split}
\end{equation}
We bound $\mathfrak{U}_1$ according to identity \eqref{consistency-qopt-velocity} (which holds in view of Lemma~\ref{L:consistency-qopt}) and inequality \eqref{DG-threshold} 
\begin{equation*}
\begin{split}
\mathfrak{U}_1
&=
\int_\Domain \GradM (u - u_\DG )\colon \Grad \Smt_\DG v 
+ \int_\Sigma \Jump{u_\DG} \cdot \Avg{\Grad v} \cdot \Normal
- \int_\Skel \dfrac{\eta}{h} \Jump{u_\DG} \cdot \Jump{v}\\
&\leq (2+ \norm{\Smt_\DG}) \norm{u-u_\DG}_\DG \norm{v}_\DG.
\end{split}
\end{equation*}
We bound $\mathfrak{U}_2$ making use of identity \eqref{quasi-optimal-nec-div} (which holds in view of Lemma~\ref{L:consistency-qopt}) and \cite[Lemma~2.1]{Pyo.Nochetto:05}
\begin{equation*}
\mathfrak{U}_2 = 
\int_\Domain (p-Rp) \Div \Smt_\DG v
\leq 
\norm{\Smt_\DG} \norm{v}_\DG \normLeb{p-Rp}{\Domain}.
\end{equation*} 
We insert the estimates of $\mathfrak{U}_1$ and $\mathfrak{U}_2$ into \eqref{quasi-optimality-velocity-proof2} and \eqref{quasi-optimality-velocity-proof3}, to obtain
\begin{equation*}
\normLeb{p-p_\DG}{\Domain}
\leq
\Visc \dfrac{2 + \norm{\Smt_\DG}}{\beta_\DG} \norm{u-u_\DG}_\DG
+
\dfrac{1+\norm{\Smt_\DG}}{\beta_\DG} 
\inf_{q \in \PolyellAvg{\degree-1}} \normLeb{p-q}{\Domain}.
\end{equation*}
We derive the claimed estimate of the pressure error by means of the previous estimate of the velocity error. 
\end{proof}

\subsection{Quasi-optimality and pressure robustness}
\label{SS:pressure-robustness}

% Quasi-optimality and pressure robustness
The assumptions in Theorem~\ref{T:quasi-optimality} do not guarantee that the discretization \eqref{Stokes-dG} is pressure robust in the sense of \eqref{qopr}. We illustrate this by a numerical experiment in section~\ref{SS:jumping-pressure}. Similarly as in \cite{Kreuzer.Zanotti:19}, we achieve pressure robustness by the additional assumption that the operator $\Smt_\DG$ preserves the discrete divergence. Recalling the definition of $\Divdisc{\DG}$ in \eqref{discrete-divergence}, this corresponds to prescribing a reinforced version of \eqref{quasi-optimal-nec-div}. 
\begin{theorem}[Quasi-optimality and pressure robustness]
\label{T:qopt-prerob}
Assume that the operator $\Smt_\DG$ satisfies \eqref{consistency-qopt-conditions} and
\begin{equation}
\label{consistency-qopt-divergence}
\Div \Smt_\DG v = \Divdisc{\DG} v
\end{equation}
for all $v \in (\Polyell{\degree})^\Dim$. Moreover, let $\eta > \overline{\eta}$, where $\overline{\eta}$ is as in \eqref{DG-threshold}. Then, denoting by $(u,p)$ and $(u_\DG, p_\DG)$ the solutions of \eqref{Stokes} and \eqref{Stokes-dG}, respectively, with viscosity $\Visc > 0$ and load $f \in \SobHD{\Domain}$, we have
\begin{equation*}
\label{qopt-prerob-velocity}
\norm{u-u_\DG}_\DG
\leq C 
\delta_\DG \dfrac{1+\norm{E_\DG}}{\underline{\alpha}_\DG}
\inf_{w \in (\Polyell{\degree})^\Dim}
\norm{u-w}_{\DG}
\end{equation*}
and
\begin{equation*}
\label{qopt-prerob-pressure}
\normLeb{p-p_\DG}{\Domain}
\leq
C \delta_\DG \dfrac{(1+\norm{\Smt_\DG})^2}{\underline{\alpha}_\DG\beta_\DG} \Visc
\inf_{w \in (\Polyell{\degree})^\Dim}
\norm{u-w}_{\DG}
+
\inf_{q \in \PolyellAvg{\degree-1}}
\normLeb{p-q}{\Domain}. 
\end{equation*}
\end{theorem} 
\begin{proof}
The proof is the same as for Theorem~\ref{T:quasi-optimality}, with the only difference that we have $\mathfrak{T}_2 = 0$ and $\mathfrak{U}_2 = 0$ in \eqref{quasi-optimality-velocity-proof} and \eqref{quasi-optimality-velocity-proof3}, respectively, as a consequence of \eqref{consistency-qopt-divergence}.
\end{proof}

\subsection{Weak jump penalization}
\label{SS:weak-penalization}

% Dependence of the constants on the penalty
According to Lemma~\ref{L:stability-velocity+pressure-DG}, the discretization \eqref{Stokes-dG} is uniquely solvable provided the penalty parameter $\eta$ is `large enough'. Therefore, it is worth checking the asymptotic behavior of the constants in the previous error estimates for $\eta \to +\infty$. To this end, recall the definition of the constant $\underline{\alpha}_\DG$ and the estimates of $\beta_\DG^{-1}$ and $\delta_\DG$ in \eqref{coercivity+boundedness+infsup-DG} and Lemma~\ref{L:discretely-divfree-approx}, respectively. Assume also that the operator norm of $\Smt_\DG$ can be bounded irrespective of $\eta$. Then, we see that the constant in the velocity error estimates of Theorems~\ref{T:quasi-optimality} and \ref{T:qopt-prerob} is $\Cleq \sqrt{\eta}$. Similarly, the constant in the corresponding pressure error estimates is $\Cleq \eta$. This indicates that we may have locking, in the sense of \cite{Babuska.Suri:92}, in the limit $\eta \to +\infty$. Moreover, the pressure error is potentially more sensitive to large values of $\eta$ than the velocity error. We confirm both expectations by a numerical experiment in section~\ref{SS:locking}.

% Full jump penalization
To be more precise, consider the $H^1_0$-conforming space
\begin{equation*}
\label{Lagrange-space}
\Lagr{\degree} := \SobH{\Domain} \cap \Polyell{\degree}
= \{ v \in \Polyell{\degree} \mid  \Jump{v} = 0 \; \text{on} \; \Skel \}
\end{equation*}  
and the subspace 
\begin{equation*}
Z_\mathrm{SV} := \{ z \in (\Lagr{1})^\Dim \mid \Div z = 0 \}.
\end{equation*}
Let $u_\DG$ be defined by \eqref{Stokes-dG}. The inclusion $u_\DG \in Z_\DG$ and the a priori estimate in Lemma~\ref{L:stability-velocity+pressure-DG} entail that $u_\DG$ converges to an element of $Z_\mathrm{SV}$ as $\eta \to +\infty$. Hence, the best constant in the velocity error estimate of Theorem~\ref{T:quasi-optimality-EDG} cannot be smaller than the best constant $\delta_\mathrm{SV} \geq 1$ in the inequality
\begin{equation}
\label{delta-SV}
\forall u \in Z \qquad
\inf_{z \in Z_\mathrm{SV}} \normLeb{\Grad(u-z)}{\Domain}
\leq \delta_{\mathrm{SV}}
\inf_{v \in (\Lagr{\degree})^\Dim} 
\normLeb{\Grad(u-v)}{\Domain}
\end{equation}
in the limit $\eta \to +\infty$. Note that the size of $\delta_\mathrm{SV}$ is intimately related to the stability of the Scott-Vogelius pair $(\Lagr{\degree})^\Dim / \Div(\Lagr{\degree})^\Dim$. Unfortunately, such constant is known to be large for various combinations of $\degree, \Dim$ and $\Mesh$, see e.g. \cite[sections 4-5]{Babuska.Suri:92b}.  

% Weak jump penalization
A possible way out consists in considering variants of the form $a_\DG$ and of the scalar product $(\cdot, \cdot)_\DG$ with
\begin{equation}
\label{weak-penalization}
\int_\Skel \dfrac{\eta}{h}\Jump{w} \cdot \Jump{v}
\qquad \text{replaced by} \qquad
\int_\Skel \dfrac{\eta}{h} \pi_{\degree-1} \Jump{w} \cdot  \pi_{\degree-1} \Jump{v}
\end{equation}
where the $L^2$-orthogonal projection $\pi_{\degree-1}$ onto the space $\Polyell{\degree-1}(\Skel) := \{ v: \Skel \to \R \mid \forall F \in \FacesM \quad v_{|F} \in \Poly{\degree-1}(F) \}$ is applied component-wise. Such a modification does not affect neither the expression of the constants in \eqref{boundedness-DG} and \eqref{coercivity-DG} nor the validity of Theorems~\ref{T:quasi-optimality} and \ref{T:qopt-prerob}. Indeed, it can easily be shown that
\begin{equation*}
\label{inf-sup-CR}
\beta_\DG \geq \beta_{\mathrm{CR}}
:=
\inf_{q \in \PolyellAvg{\degree-1}} 
\sup_{w \in (\CR{\degree})^\Dim}
\dfrac{\int_\Domain q \DivM w}{\normLeb{\GradM w}{\Domain} \normLeb{q}{\Domain}}
\end{equation*}
where 
\begin{equation*}
\label{CR-space}
\CR{\degree} := \{ v \in  \Polyell{\degree} \mid \pi_{\degree-1}\Jump{v} =0 \; \text{on} \; \Skel \}.
\end{equation*}
With this modification, the constants in Theorems~\ref{T:quasi-optimality} and \ref{T:qopt-prerob} are bounded irrespective of $\eta$, provided $\beta_{\mathrm{CR}}^{-1} \leq C$. Several results ensuring the validity of this condition, for various combinations of $\degree, \Dim$ and $\Mesh$, are available in the literature, see \cite{Crouzeix.Raviart:73,Crouzeix.Falk:89,Baran.Stoyan:07} and the references therein. Moreover, we are not aware of any negative result.

\section{A moment- and divergence-preserving operator}
\label{S:smoother}
Motivated by the error estimates in Theorem~\ref{T:qopt-prerob}, we now aim at designing a linear operator $\Smt_\DG: (\Polyell{\degree})^\Dim \to \SobH{\Domain}^\Dim$ which satisfies the following conditions
\begin{subequations}
\label{EDG-conditions}
\begin{align}
\label{EDG-conditions-stability}
&\Smt_\DG \text{ is stable, in that } \norm{\Smt_\DG} \leq C,\\
\label{EDG-conditions-facemom}
&\Smt_\DG \text{ preserves } \Poly{\degree-1}(F)^\Dim \text{-moments on each } F \in \FacesMint, \text{ see } \eqref{consistency-qopt-faces}, \\
\label{EDG-conditions-divergence}
&\Smt_\DG \text{ preserves the discrete divergence } \Divdisc{\DG}, \text{ see } \eqref{consistency-qopt-divergence},\\
\label{EDG-conditions-elemmom}
&\Smt_\DG \text{ preserves } \Poly{\degree-2}(K)^\Dim \text{-moments in each } K \in \Mesh, \text{ see }
 \eqref{consistency-qopt-simplices}.
\end{align}
\end{subequations} 
We restrict ourselves to the case \(d=2\), in order to keep the discussion as easy as possible. In section~\ref{sec:3d-smoother}, we
discuss the differences in the design for \(d=3\).

\subsection{Outline of the construction}
\label{SS:ouline}

We first outline the strategy underlying our construction before we enter into the technical details. We shall obtain $\Smt_\DG$ from the combination of four operators, namely
\begin{equation}
\label{EDG-definition}
\Smt_\DG := \Smt_1 + \Smt_2 + \Smt_3 + \Smt_4.
\end{equation}
Our construction has a recursive structure in the sense that the definition of $\Smt_i$, $i \in \{2, \dots, 4\}$, involves the one of $\Smt_1, \dots, \Smt_{i-1}$. The role of each summand in \eqref{EDG-definition} can be summarized as follows.  
\begin{itemize}
	\item The first operator $\Smt_1$ maps $(\Polyell{\degree})^2$ into $(\Polyell{\degree} \cap \SobH{\Domain})^2$ by a simple averaging technique and is stable, in that \eqref{EDG-conditions-stability} holds.
	\item The second operator preserves the stability of $\Smt_1$, while correcting the moments on faces. This is obtained by mapping into a space of face-bubbles. As a result, the sum $\Smt_1 + \Smt_2$ enjoys both \eqref{EDG-conditions-stability} and \eqref{EDG-conditions-facemom}.
	\item The third operator $\Smt_3$ additionally enforces \eqref{EDG-conditions-divergence}, while preserving the validity of the previous properties. This is achieved by mapping into a space of volume-bubbles.
	\item Finally, the fourth operator $\Smt_4$ maps into a space of divergence-free volume-bubbles and is designed to guarantee that $\Smt_\DG$ enjoys also the last condition prescribed in \eqref{EDG-conditions-elemmom}.  
\end{itemize}
       
For $v \in (\Polyell{\degree})^2$, the definition of $\Smt_1$ in a simplex $K \in \Mesh$ involves the values of $v$ in the star around $K$, i.e. in the neighbouring simplices. The operators $\Smt_2$, $\Smt_3$ and $\Smt_4$ are obtained solving local problems on the faces or on the simplices of $\Mesh$. Each local problem is independent of the others and can efficiently be solved by resorting to a reference configuration. Thus, the resulting operator $\Smt_\DG$ is \textit{computationally feasible}, in the sense that, for any nodal basis function $\Phi$ of $(\Polyell{\degree})^2$, the computation of $\Smt_\DG \Phi$ requires only  $\mathrm{O}(1)$ operations.    

\subsection{Preliminary observations}
\label{SS:technical-preliminaries}

The main difficulty in the construction of $\Smt_\DG$ is that
conditions \eqref{EDG-conditions-facemom}, \eqref{EDG-conditions-divergence} and \eqref{EDG-conditions-elemmom} are not linearly independent. In fact,
prescribing sufficiently many moments of $\Smt_\DG v$ on the skeleton of $\Mesh$ as well as the divergence of $\Smt_\DG v$ can be expected to prescribe implicitly
also the moments of $\Smt_\DG v$ times gradients on each simplex of $\Mesh$. The next lemma
states this observation more precisely, showing also that the above conditions are at least compatible.   
\begin{lemma}[$\Grad \Poly{\degree-1}(K)$-moments]
\label{lem:DIV=>KMoments}
Let $\Smt: (\Polyell{\degree})^2 \to \SobH{\Domain}^2$ be an operator fulfilling \eqref{EDG-conditions-facemom} and \eqref{EDG-conditions-divergence}. Then, for all 
$v \in (\Polyell{\degree})^2$, $K \in \Mesh$ and $q \in \Poly{\degree-1}(K)$, we~have
\begin{equation}\label{element-moments-sideeffects}
\int_K \Smt v \cdot \Grad q = \int_K v \cdot \Grad q.
\end{equation}
\end{lemma}
\begin{proof}
Let $v \in (\Polyell{\degree})^2$ and $q \in \Poly{\degree-1}(K)$ be given. We extend $q$ to $\Domain \setminus K$ by zero. The integration by parts formula \eqref{ibp-elementwise} yields
\begin{equation*}
\begin{split}
\int_K \Smt v \cdot \Grad q
&=
-\int_\Domain q \Div \Smt v 
+ \int_{\Skel \setminus \partial \Domain} \Smt v \cdot \Normal \Jump{q}\\
&=
-\int_\Domain q \Divdisc{\DG} v 
+ \int_{\Skel \setminus \partial \Domain} \Avg{v} \cdot \Normal \Jump{q}, 
\end{split}
\end{equation*}
where the second identity follows from the assumption that $\Smt$ satisfies \eqref{EDG-conditions-facemom} and \eqref{EDG-conditions-divergence}. The equivalent definition of the discrete divergence in Remark~\ref{R:discrete-divergence} entails that
\begin{equation*}
\int_K \Smt v \cdot \Grad q
=
-\int_\Domain q \DivM v 
+ \int_\Skel \Jump{v} \cdot \Normal \Avg{q}
+ \int_{\Skel \setminus \partial \Domain} \Avg{v} \cdot \Normal \Jump{q}. 
\end{equation*} 
We conclude invoking once again the element-wise integration by parts formula and recalling that $q$ vanishes in $\Domain \setminus K$. 
\end{proof}

The above lemma suggests that we should enforce \eqref{EDG-conditions-elemmom} only on some complement of $\Grad \Poly{\degree-1}(K)$ in $\Poly{\degree-2}(K)^2$. We shall identify one such complement and construct $\Smt_\DG$ with the help of the \textit{curl} and \textit{rot} operators, that are defined as
\begin{equation}
\label{curl-rot-d=2}
\Curl(w) := (\partial_2 w, -\partial_1 w)
\qquad \text{and} \qquad
\Rot(v) := -\partial_2 v_1 + \partial_1 v_2
\end{equation}
where $w$ and $v=(v_1, v_2)$, respectively, are scalar- and vector-valued functions. Recall that assuming $w \in \SobH{K}$ and $v \in \Sob{K}$, we have 
\begin{equation}
\label{ibp-d=2}
\int_K \Curl(w) \cdot v
=
\int_K w \Rot(v)
\end{equation}
for all $K \in \Mesh$. Moreover, it holds
\begin{equation}
\label{div-curl=0}
\Div(\Curl (v)) =0.
\end{equation}

Recall the convention $\Poly{-1} = \{0\}$. The next lemma provides the desired decomposition of $\Poly{\degree-2}(K)^2$. 

\begin{lemma}[Decomposition of vector-valued polynomials]
\label{L:decompositions}
For all $k \geq 0$ and $K \in \Mesh$, define 
\begin{equation*}
x^\perp \Poly{k-1}(K)
:=
\left \{ x^\perp r := (-x_2 r, x_1 r)
\mid  r \in \Poly{k-1}(K) \right\}.
\end{equation*}
The operator $\Rot: \Poly{k}(K)^2 \to \Poly{k-1}(K)$ is injective on $x^\perp \Poly{k-1}(K)$ and its kernel coincides with $\Grad \Poly{k+1}(K)$. As a consequence, we have
\begin{equation}
\label{decompositions-d=2}
\Poly{k}(K)^2
=
\Grad \Poly{k+1}(K) 
\oplus
x^\perp \Poly{k-1}(K).
\end{equation}
\end{lemma}

\begin{proof}
We assume $k \geq 1$, because the claim is clear for $k = 0$. Let $\Rot(x^\perp r) = 0$ for some $r = \sum_{\snorm{\alpha} \leq k-1} a_\alpha x^\alpha \in \Poly{k-1}(K)$. We infer $\sum_{\snorm{\alpha} \leq k-1} (2+\snorm{\alpha}) a_\alpha x^\alpha = 0$, showing that $r = 0$. This proves the injectivity of $\Rot$ on $x^\perp \Poly{k-1}(K)$. Next, the fact that the kernel of $\Rot$ on $\Poly{k}(K)^2$ coincides with $\Grad\Poly{k+1}(K)$ is a standard result from vector calculus. This entails that $\Grad \Poly{k+1}(K) \cap x^\perp \Poly{k-1}(K) = \{0\}$. Then, the claimed decomposition of $\Poly{k}(K)^2$ follows from a dimensional argument.     
\end{proof}

As mentioned before, the construction of the operators $\Smt_3$ and $\Smt_4$ in \eqref{EDG-definition} involves the solution of local problems on each triangle in $\Mesh$. For both theoretical and computational convenience, we shall formulate such problems on a reference triangle $\RefSim$, with the help of the Piola's transformations, see e.g. \cite[Section~2.1.3]{Boffi:Brezzi:Fortin.13}. Hence, for all $K \in \Mesh$, we fix
a one-to-one affine mapping $F_K: \RefSim \to K$, with Jacobian matrix
$DF_K$. We set $J_K := \snorm{\det DF_K}$. Note that $DF_K$
is a constant invertible matrix and that $J_K$ is a positive constant. 

The contravariant and the covariant Piola's transformations, respectively, map functions $v_\Ref, w_\Ref \in \Sob{\RefSim}^2$ into $\Sob{K}^2$ and are given by   
\begin{equation}
\label{Piola}
\PiolaCon v_\Ref 
:=
J_K^{-1} DF_K (v_\Ref \circ F_K^{-1}) 
\quad \text{and} \quad
\PiolaCov w_\Ref
:=
DF_K^{-T} (w_\Ref \circ F_K^{-1}).
\end{equation}
Remarkably, we have
\begin{equation}
\label{Piola-integral}
\int_K \PiolaCon v_\Ref \cdot \PiolaCov w_\Ref
=
\int_{\RefSim} v_\Ref \cdot w_\Ref.
\end{equation}
Moreover, the contravariant Piola's transformation is such that 
\begin{equation}
\label{Piola-divergence}
\Div(\PiolaCon v_\Ref)
=
J_K^{-1} (\Div v_\Ref) \circ F_K^{-1}
\end{equation}
and
\begin{equation}
\label{Piola-scaling}
\normLeb{\PiolaCon v_\Ref}{K} 
\leq
J_K^{-\frac{1}{2}} h_K 
\normLeb{v_\Ref}{\RefSim}
\leq
C
\normLeb{\PiolaCon v_\Ref}{K} .
\end{equation}

\subsection{Construction of $\Smt_\DG$} 
\label{SS:construction-EDG}
We now construct an operator $\Smt_\DG: (\Polyell{\degree})^\Dim \to \SobH{\Domain}^\Dim$, $\degree \in \N$, which satisfies \eqref{EDG-conditions}. As stated in \eqref{EDG-definition}, we set $\Smt_\DG := \sum_{i=1}^4 \Smt_i$, where each operator $\Smt_i$ is defined as follows. \medskip

\emph{Definition of $\Smt_1$}. Each polynomial in
\(\Poly{\degree}(K)\), $K \in \Mesh$, is uniquely determined by its point values at the
Lagrange nodes \(\mathcal{L}_{\ell}(K)\) of degree
\(\degree\). Recall also that the nodal degrees of freedom
of \(\Lagr{\degree}=\Polyell{\degree}\cap H_0^1(\Omega)\) are given by the evaluations at the points \(\mathcal{L}_\degree:=\bigcup_{K\in
  \Mesh}\mathcal{L}_\degree(K)\cap\Omega\). We denote by
\(\Phi^z_\degree\in \Lagr{\degree}\) the nodal basis
function associated with the evaluation at \(z\in \mathcal{L}_\degree\), that is \(\Phi^z_\degree(y)=\delta_{zy}\) for all \(y,z\in
\mathcal{L}_\degree\). Then, for $v \in (\Polyell{\degree})^\Dim$, we let $\Smt_1 v$ be defined by 
\begin{equation}
\label{E1-definition}
\Smt_1 v :=\sum_{z\in \mathcal{L}_\degree} \dfrac{1}{N_z}
\left( \sum_{K \in \Mesh, \,K \ni z } v_{|K}(z) \right) \Phi^z_\degree
\end{equation}
where $N_z$ is the number of triangles in $\Mesh$ touching $z$. Averaging operators like $\Smt_1$ or variants are common devices in the context of dG methods, see e.g. \cite[section~5.5.2]{DiPietro.Ern:12}.
\medskip

\emph{Definition of $\Smt_2$}. We define $\Smt_2$ in the vein of
\cite[Section~3.2]{Veeser.Zanotti:18b}. For every interior edge \(F \in \FacesMint\), let $K_1, K_2 \in \Mesh$ be such that \(F =  K_1 \cap K_2\). Denote by \(\mathcal{L}_\degree(F):=\mathcal{L}_\degree(K_1)\cap \mathcal{L}_\degree(K_2) \) the Lagrange nodes of degree $\degree$ on $F$ and let \(b_F:=\prod_{z\in \mathcal{L}_1(F)}\Phi^z_1\) be the quadratic face bubble supported on $K_1 \cup K_2$. We introduce a linear operator
\(\Smt_{2,F}:L^2(F)^2\to \Poly{\degree-1}(F)^2\) by solving the local problem
\begin{align}
\label{E2F-definition}
\forall m_F \in \Poly{\degree-1}(F)^2\qquad\int_F \Smt_{2,F} v \cdot m_F\,b_F=\int_Fv \cdot m_F.
\end{align}
Then, for \(v \in (\Polyell{\degree})^2 \), we set
\begin{equation}
\label{E2-definition}
\Smt_2 v :=\sum_{F\in\FacesMint}\sum_{z\in\mathcal{L}_{\degree-1}(F)}
\Smt_{2, F}(\Avg{v} - \Smt_1 v)(z)\,\Phi^z_{\degree-1} b_F.
\end{equation}
Note that each summand involves an extension from $F$ to $K_1 \cup K_2$.
\medskip

\emph{Definition of $\Smt_3$}. We define $\Smt_3$ in the vein of \cite{Kreuzer.Zanotti:19,Verfuerth.Zanotti:19}. Let $\RefSim$ be the reference triangle introduced in section~\ref{SS:technical-preliminaries}. We obtain a triangulation $\Mesh_\Ref$ of $\RefSim$ connecting each vertex with the barycenter, see Figure~\ref{F:alfeld-refinement}. The space $\Polyell{\degree+1}(\Mesh_\Ref)$ consists of all piecewise polynomials of degree $\leq (\degree+1)$ on $\Mesh_\Ref$. We consider the subspaces
\begin{equation*}
\Lagr{\degree+1}(\Mesh_\Ref)
:= 
\Polyell{\degree+1}(\Mesh_\Ref) \cap \SobH{\RefSim}
\quad \text{and} \quad
\PolyellAvg{\degree}(\Mesh_\Ref)
:=
\Polyell{\degree}(\Mesh_\Ref) \cap \LebH{\RefSim}
\end{equation*}
and introduce a linear operator $\Smt_{3,\Ref}: \PolyellAvg{\degree}(\Mesh_\Ref) \to \Lagr{\degree+1}(\Mesh_\Ref)^2$ by imposing
\begin{equation}
\label{E3ref-definition}
\Smt_{3,\Ref} (q_\Ref) := \mathrm{argmin} \left \{ \normLeb{\Grad v_\Ref}{\RefSim}^2 \mid v_\Ref \in \Lagr{\degree+1}(\Mesh_\Ref)^2,  \Div v_\Ref = q_\Ref \right \}.
\end{equation} 
This constrained quadratic minimization problem is uniquely solvable as a consequence of \cite[Theorem~3.1]{Guzman.Neilan:18}. Note that we can equivalently rewrite \eqref{E3ref-definition} as a discrete Stokes-like problem, with velocity space $\Lagr{\degree+1}(\Mesh_\Ref)^2$, pressure space $\PolyellAvg{\degree}(\Mesh_\Ref)$ and right-hand side zero in the momentum equation and $q_\Ref$ in the continuity equation. Then, for $v \in (\Polyell{\degree})^2$, we define
\begin{equation}
\label{E3-definition}
\Smt_3 v :=
\sum_{K \in \Mesh}  \PiolaCon \Smt_{3,\Ref}( J_K (\Divdisc{\DG} v - \sum_{i=1}^{2}\Div \Smt_i v) \circ F_K)
\end{equation}
where each summand vanishes on $\partial K$ and is extended by zero outside $K$. The discussion in the next section confirms that the argument of $\Smt_{3, \Ref}$ is indeed an element of $\PolyellAvg{\degree}(\Mesh_\Ref)$. \medskip

\emph{Definition of $\Smt_4$}. Denote by $b_{\Ref}$ the cubic
bubble function on $\RefSim$, that is obtained by taking the product of the Lagrange basis functions of \(\Poly{1}(\RefSim)\) associated with the evaluations at the vertices of $\RefSim$. For $\degree \geq 3$, we introduce a linear operator $\Smt_{4, \Ref}: \Leb{\RefSim}^2 \to x^\perp \Poly{\degree-3}(\RefSim)$ by imposing
\begin{equation}
\label{E4ref-definition}
\forall m_\Ref \in x^\perp \Poly{\degree-3}(\RefSim) \quad\int_{\RefSim} \Rot(\Smt_{4, \Ref} q_\Ref) \Rot(m_\Ref) b_{\Ref}^2
=
\int_{\RefSim} q_\Ref \cdot m_\Ref. 
\end{equation}
Lemma~\ref{L:decompositions} ensures that this problem is uniquely solvable. Then, for $v \in (\Polyell{\degree})^2$, we define $\Smt_4 v = 0$ if $\degree \in  \{1,2\}$, otherwise 
\begin{equation}
\label{E4-definition}
\Smt_4 v := \sum_{K \in \Mesh}
\PiolaCon \Curl( b_{\Ref}^2 \Rot \Smt_{4, \Ref} (\PiolaCon)^{-1}(v - \sum_{i=1}^{3}\Smt_i v)  )
\end{equation} 
where each summand vanishes on $\partial K$ and is extended by zero outside $K$. 

% Figure with Alfeld refinement
\begin{figure}[ht]
	\centering
	\begin{tikzpicture}
	% Coordinates of the first triangle
	\coordinate (z1) at (0,0);
	\coordinate (z2) at (2,0);
	\coordinate (z3) at (1,1.5);
	\coordinate (c1) at (1, 0.5);
	% Draw the edges of the first triangle
	\path (z1) edge (z2);
	\path (z2) edge (z3);
	\path (z3) edge (z1);
	% Coordinates of the second triangle
	\coordinate (z4) at (4,0);
	\coordinate (z5) at (6,0);
	\coordinate (z6) at (5,1.5);
	\coordinate (c2) at (5, 0.5);
	% Draw the edges of the second triangle
	\path (z4) edge (z5);
	\path (z5) edge (z6);
	\path (z6) edge (z4);
	\path[dashed] (z4) edge (c2);
	\path[dashed] (z5) edge (c2);
	\path[dashed] (z6) edge (c2);
	\end{tikzpicture}
	\caption{Reference triangle $\RefSim$ (left) and triangulation $\Mesh_\Ref$ (right).}
	\label{F:alfeld-refinement}
\end{figure}
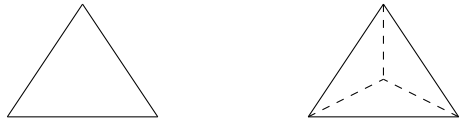

\subsection{Preservation properties of $\Smt_\DG$}
\label{SS:preservation-EDG}
In this section we prove that the operator $\Smt_\DG$ defined above satisfies the conditions \eqref{EDG-conditions-facemom}, \eqref{EDG-conditions-divergence} and \eqref{EDG-conditions-elemmom}, i.e. it preserves the discrete divergence and all the prescribed moments on the faces and the triangles of $\Mesh$.
For this purpose, we make use of the
following integration by parts formula, which generalizes Lemma~\ref{lem:DIV=>KMoments}. 

\begin{lemma}
\label{L:loc-Gauss}
Let $\Smt: (\Polyell{\degree})^2 \to \SobH{\Domain}^2$ be a linear operator satisfying \eqref{EDG-conditions-facemom}. Then,  
for all $v \in (\Polyell{\degree})^2$, \(K\in\Mesh\) and
\(q\in\Poly{\ell-1}(K)\), we have
\begin{align*}
\int_K
(\Divdisc{\DG}v - \Div \Smt v) q=- \int_K
    (v - \Smt v) \cdot \nabla q.
  \end{align*}
\end{lemma}

\begin{proof}
Proceed as in the proof of Lemma~\ref{lem:DIV=>KMoments}, without assuming that $\Smt$ satisfies condition \eqref{EDG-conditions-divergence}.
\end{proof}

We are now prepared to prove the claimed properties of $\Smt_\DG$.

\begin{theorem}[Preservation properties of $\Smt_\DG$]
\label{T:conservation}
The operator $\Smt_\DG$ defined in section~\ref{SS:construction-EDG}
satisfies the conditions \eqref{EDG-conditions-facemom}, \eqref{EDG-conditions-divergence} and \eqref{EDG-conditions-elemmom}.
\end{theorem}

\begin{proof}
Let $v \in (\Polyell{\degree})^2$. We check one by one the validity of the desired conditions.\medskip

\textit{Proof of \eqref{EDG-conditions-facemom}}. By construction, each summand in the definitions \eqref{E3-definition} and \eqref{E4-definition} of $\Smt_3$ and $\Smt_4$, respectively, is supported in one triangle $K \in \Mesh$ and vanishes on $\partial K$. This entails that $\Smt_\DG v = \Smt_1 v + \Smt_2 v$ on the skeleton $\Skel$. Moreover, for $F \in \FacesMint$, we have $(\Smt_2 v)_{|F} = \Smt_{2,F} (\Avg{v} - \Smt_1 v) b_F$, as a consequence of \eqref{E2-definition}. Hence, for all $m_F \in \Poly{\degree-1}(F)^2$, the definition of $\Smt_{2,F}$ in \eqref{E2F-definition} implies
\begin{equation*}
\int_F \Smt_2 v \cdot m_F
=
\int_F \Smt_{2,F} (\Avg{v} - \Smt_1 v) \cdot m_F b_F
=
\int_F (\Avg{v} - \Smt_1 v) \cdot m_F.
\end{equation*}
Rearranging terms, we infer that
\begin{equation*}
\int_F \Smt_\DG v \cdot m_F
=
\int_F (\Smt_1 v + \Smt_2 v) \cdot m_F
=
\int_F \Avg{v} \cdot m_F.
\end{equation*}

\textit{Proof of \eqref{EDG-conditions-divergence}}. Each summand in the definition \eqref{E4-definition} of $\Smt_4$ is divergence-free, as a consequence of \eqref{div-curl=0} and \eqref{Piola-divergence}. This entails that $\Div \Smt_\DG v = \sum_{i=1}^{3} \Div \Smt_i v$ in $\Domain$. Moreover, for $K \in \Mesh$, the identity \eqref{Piola-divergence} and the definitions \eqref{E3ref-definition} and \eqref{E3-definition} of $\Smt_{3,\Ref}$ and $\Smt_3$, respectively, reveal that
\begin{equation*}
(\Div \Smt_3 v)_{|K} = (\Div_\DG v - \sum_{i=1}^{2} \Div \Smt_i v)_{|K}.
\end{equation*} 
Rearranging terms, we infer that
\begin{equation*}
\Div \Smt_\DG v
=
\sum_{i=1}^{3} \Div \Smt_i v
=
\Divdisc{\DG} v. 
\end{equation*}

\textit{Proof of \eqref{EDG-conditions-elemmom}}. For all $K \in \Mesh$, the covariant Piola's transformation $\PiolaCov$ from \eqref{Piola} maps $\Poly{\degree-2}(\RefSim)^2$ into $\Poly{\degree-2}(K)^2$ and is one-to-one. Then, according to the transformation rule \eqref{Piola-integral}, we see that the following identity
\begin{equation}
\label{EDG-conditions-elemmom-equiv}
\forall m_\Ref \in \Poly{\degree-2} (\RefSim)^2
\quad
\int_{\RefSim} (\PiolaCon)^{-1}\Smt_\DG v \cdot m_\Ref
=
\int_{\RefSim} (\PiolaCon)^{-1}v \cdot m_\Ref
\end{equation}
is an equivalent formulation of \eqref{EDG-conditions-elemmom}. Moreover, according to the decomposition stated in Lemma~\ref{L:decompositions}, we can split \eqref{EDG-conditions-elemmom-equiv} into two independent conditions with test functions in $\Grad \Poly{\degree-1}(\RefSim)$ and $x^\perp \Poly{\degree-3}(\RefSim)$, respectively. Let us first assume that $m_\Ref = \Grad q_\Ref \in \Grad \Poly{\degree-1}(\RefSim)$. The definition of $\Smt_4$ in \eqref{E4-definition}, the integration by parts rule \eqref{ibp-d=2} and Lemma~\ref{L:decompositions} imply that
\begin{equation*}
\begin{split}
&\int_{\RefSim} (\PiolaCon)^{-1} \Smt_4 v \cdot m_\Ref
=
\int_{\RefSim} \Curl ( b_{\Ref}^2 \Rot \Smt_{4, \Ref} (\PiolaCon)^{-1}(v - \sum_{i=1}^{3}\Smt_i v) ) \cdot  \Grad q_\Ref \\
&\qquad =
\int_{\RefSim} \Rot \Smt_{4, \Ref} (\PiolaCon)^{-1}(v - \sum_{i=1}^{3}\Smt_i v)  \Rot( \Grad q_\Ref) b_{\Ref}^2
= 0.
\end{split}
\end{equation*}
Next, recall the definitions of $\Smt_{3, \Ref}$ and $\Smt_3$ in \eqref{E3ref-definition} and \eqref{E3-definition}, respectively. Integrating by parts, changing variables twice and invoking Lemma~\ref{L:loc-Gauss}, we obtain
\begin{equation*}
\begin{split}
&\int_{\RefSim} (\PiolaCon)^{-1} \Smt_3 v \cdot m_\Ref
= \int_{\RefSim} \Smt_{3, \Ref} v \cdot \Grad q_\Ref\\
&\qquad=-\int_{\RefSim} (\Div \Smt_{3, \Ref} v) q_\Ref 
=-\int_{\RefSim} J_K(\Divdisc{\DG} v - \sum_{i=1}^{2} \Div \Smt_i v)\circ F_K \:q_\Ref\\
&\qquad=-\int_K (\Divdisc{\DG} v - \sum_{i=1}^{2} \Div \Smt_i v) q_\Ref \circ F_{K}^{-1}
=-\int_K (v- \sum_{i=1}^{2}\Smt_i v) \cdot \Grad(q_\Ref \circ F_K^{-1})\\
&\qquad=\int_{\RefSim} (\PiolaCon)^{-1} (v-\sum_{i=1}^{2}\Smt_i v) \cdot m_\Ref.
\end{split}
\end{equation*} 
Combining this identity with the previous one and rearranging terms, we infer that \eqref{EDG-conditions-elemmom-equiv} holds for all $m_\Ref \in \Grad \Poly{\degree-1}(\RefSim)$. This concludes the proof for $\degree \in \{1,2\}$. For $\degree \geq 3$, assume further $m_\Ref \in x^\perp \Poly{\degree-3}(\RefSim)$. The definitions of $\Smt_{4, \Ref}$ and $\Smt_4$ in \eqref{E4ref-definition} and \eqref{E4-definition}, respectively, and the integration by parts rule \eqref{ibp-d=2} reveal that
\begin{equation*}
\begin{split}
\int_{\RefSim} (\PiolaCon)^{-1} \Smt_4 v \cdot m_\Ref
&=
\int_{\RefSim} \Rot( \Smt_{4, \Ref} (\PiolaCon)^{-1}(v - \sum_{i=1}^{3} \Smt_i v) ) \Rot(m_\Ref) b_\Ref^2\\
&=\int_{\RefSim} (\PiolaCon)^{-1}(v - \sum_{i=1}^{3} \Smt_i v) \cdot m_\Ref.
\end{split}
\end{equation*} 
Rearranging terms, we infer that \eqref{EDG-conditions-elemmom-equiv} holds for all $m_\Ref \in x^\perp \Poly{\degree-3}(\RefSim)$. Thus, Lemma~\ref{L:decompositions} and the above discussion entail that $\Smt_\DG$ satisfies condition \eqref{EDG-conditions-elemmom}.
\end{proof}

\subsection{Stability of $\Smt_\DG$}
\label{SS:stability-EDG}
In this section we prove that the operator $\Smt_\DG$ defined in section~\ref{SS:construction-EDG} satisfies condition \eqref{EDG-conditions-stability}, i.e. it is stable in the norm $\norm{\cdot}_\DG$ and its stability constant $\norm{\Smt_\DG}$ is bounded in terms of the shape constant $\Shape_\Mesh$ of $\Mesh$ and of the polynomial degree \(\degree\). We begin by recalling a standard result concerning the operator $\Smt_1$ defined in \eqref{E1-definition}, see e.g., \cite[section~5.5.2]{DiPietro.Ern:12}.

\begin{lemma}[Local $L^2$-estimate of $\Smt_1$]
\label{L:averaging-estimate}
For all $v \in (\Polyell{\degree})^2$ and $K \in \Mesh$, we have
\begin{equation*}
\label{averaging-estimate}
\normLeb{v-\Smt_1 v}{K}
\leq C
\sum_{F \in \FacesM, F \cap K \neq \emptyset}
h_F^{\frac{1}{2}}
\normLeb{\Jump{v}}{F}.
\end{equation*}
\end{lemma}

Next, we prove that $\Smt_\DG$ enjoys the same local estimate as $\Smt_1$, possibly up to a different constant.  

\begin{proposition}[Local $L^2$-estimate of $\Smt_\DG$]
	\label{P:stability-loc}
	The operator $\Smt_\DG$ defined in section~\ref{SS:construction-EDG} is such that, for all $v \in (\Polyell{\degree})^2$ and $K \in \Mesh$,
        \begin{align}
        \label{stability-loc}
         \normLeb{v-\Smt_\DG v}{K}
         \leq C
         \sum_{F \in \FacesM, F \cap K \neq \emptyset}
         h_F^{\frac{1}{2}}
         \normLeb{\Jump{v}}{F}.
        \end{align} 
\end{proposition}

\begin{proof}
First of all, we recall that $\Smt_\DG = \sum_{i=1}^4 \Smt_i$ and apply the triangle inequality
\begin{equation*}
\normLeb{v-\Smt_\DG v}{K}
\leq
\normLeb{v-\Smt_1 v}{K}
+
\sum_{i=2}^4 \normLeb{\Smt_i v}{K}.
\end{equation*}
According to Lemma~\ref{L:averaging-estimate}, we only need to bound the last three summands in the right-hand side. We estimate these terms one by one. \medskip

\textit{Estimate of $\Smt_2$}. The definition of $\Smt_2$ in \eqref{E2-definition} and standard scaling arguments imply that
\begin{align*}
\normLeb{\Smt_2 v}{K}
&\leq C 
\sum_{F \in \FacesMint, F \subseteq \partial K} \snorm{K}^{\frac{1}{2}} \sum_{z \in \mathcal{L}_{\degree-1}(F)}
\snorm{\Smt_{2,F}(\Avg{v}-\Smt_1 v)(z)}\\
&\leq C \sum_{F \in \FacesMint, F \subseteq \partial K}
h_F^{\frac{1}{2}}
\normLeb{\Smt_{2,F}(\Avg{v}- \Smt_1 v) b_F^{\frac{1}{2}}}{F}.
\end{align*}
Recalling also the definition of $\Smt_{2,F}$ in \eqref{E2F-definition}, we infer that
\begin{equation*}
\normLeb{\Smt_{2,F}(\Avg{v}- \Smt_1 v) b_F^{\frac{1}{2}}}{F}
\leq
\normLeb{\Avg{v}- \Smt_1 v}{F}
\end{equation*}
for all $F \in \FacesMint$ with $F \subseteq \partial K$. 
We insert this estimate into the previous one. Then, we observe that $\snorm{\Avg{v} - v_{|K}} = \frac{1}{2}\snorm{\Jump{v}}$ on each face $F$ involved in the above summation. This fact and an inverse trace inequality entail that
\begin{equation*}
\normLeb{\Smt_2 v}{K}
\leq C (  \normLeb{v-\Smt_1 v}{K}
+
\sum_{F \in \FacesMint, F \subseteq \partial K} h_F^{\frac{1}{2}} \normLeb{\Jump{v}}{F} ).
\end{equation*}
Then, Lemma~\ref{L:averaging-estimate} yields
\begin{equation}
\label{stability-loc-E2}
\normLeb{\Smt_2 v}{K} 
\leq C
\sum_{F \in \FacesM, F \cap K \neq \emptyset}
h_F^{\frac{1}{2}}
\normLeb{\Jump{v}}{F}.
\end{equation}

\textit{Estimate of $\Smt_3$}. The definition of $\Smt_3$ in \eqref{E3-definition} and the transformation rule \eqref{Piola-scaling} imply that
\begin{equation*}
\normLeb{\Smt_3 v}{K}
\leq J_K^{-\frac{1}{2}} h_K
 \normLeb{\Smt_{3,\Ref}(J_K(\Divdisc{\DG} v - \sum_{i=1}^2 \Div \Smt_i v)\circ F_K)}{\RefSim}.
\end{equation*}
Since $\Smt_{3,\Ref}$ is a linear operator defined on a finite-dimensional space, it is bounded. We combine this observation with a change of variables
\begin{equation*}
\normLeb{\Smt_3 v}{K}
\leq C h_K
\normLeb{\Divdisc{\DG} v - \sum_{i=1}^{2} \Div \Smt_i v}{K}.
\end{equation*}
The inclusion $\Smt_1 v \in (\Polyell{\degree} \cap \SobH{\Domain})^2$ reveals that $\Div \Smt_1 v = \Divdisc{\DG} \Smt_1 v$. Recalling the equivalent definition of $\Divdisc{\DG}$ in Remark~\ref{R:discrete-divergence}, we obtain
\begin{align*}
\normLeb{\Divdisc{\DG}( v-\Smt_1 v) }{K}
&\leq C (\normLeb{\Div (v-\Smt_1)}{K} +  
\sum_{F \in \FacesM, F \subseteq \partial K}h_F^{-\frac{1}{2}}\normLeb{\Jump{v}}{F} )
\end{align*}
where we have made use also of the identity $\Jump{\Smt_1 v} = 0$ on
$\Skel$ and of the inverse inequality \eqref{DG-threshold}. We combine
this bound with the previous one and apply twice the inverse inequality $\normLeb{\Div \cdot}{K} \leq C h_K^{-1} \normLeb{\cdot}{K}$. This entails that
\begin{align*}
\normLeb{\Smt_3 v}{K}
\leq C
( \normLeb{v-\Smt_1 v}{K}
+ \normLeb{\Smt_2 v}{K}
+ \sum_{F \in \FacesM, F \subseteq \partial K}h_F^{\frac{1}{2}}\normLeb{\Jump{v}}{F} ).
\end{align*}
Then, Lemma~\ref{L:averaging-estimate} and inequality \eqref{stability-loc-E2} yield
\begin{equation}
\label{stability-loc-E3}
\normLeb{\Smt_3 v}{K} 
\leq C
\sum_{F \in \FacesM, F \cap K \neq \emptyset}
h_F^{\frac{1}{2}}
\normLeb{\Jump{v}}{F}.
\end{equation}

\textit{Estimate of $\Smt_4$}.  Let $K \in \Mesh$. The definition of $\Smt_4$ in \eqref{E4-definition} and the transformation rule \eqref{Piola-scaling} imply that
\begin{equation*}
\normLeb{\Smt_4 v}{K}
\leq J_K^{-\frac{1}{2}} h_K
\normLeb{\Curl( b_{\Ref}^2 \Rot \Smt_{4, \Ref} (\PiolaCon)^{-1}(v - \sum_{i=1}^{3}\Smt_i v)  )}{\RefSim}.
\end{equation*}
Since $\Smt_{4,\Ref}$ is a linear operator defined on a finite-dimensional space, it is bounded. We combine this observation with an inverse estimate, the transformation rule \eqref{Piola-scaling} and the triangle inequality
\begin{equation*}
\begin{split}
\normLeb{\Smt_4 v}{K}
&\leq C J_K^{-\frac{1}{2}} h_K
\normLeb{(\PiolaCon)^{-1}(v - \sum_{i=1}^{3}\Smt_i v)}{\RefSim}\\
&\leq C ( \normLeb{v-\Smt_1 v}{K} + \sum_{i=2}^{3} \normLeb{\Smt_i v}{K}).
\end{split}
\end{equation*}
Then, Lemma~\ref{L:averaging-estimate} and inequalities \eqref{stability-loc-E2} and \eqref{stability-loc-E3} yield
\begin{equation*}
\normLeb{\Smt_3 v}{K} 
\leq C
\sum_{F \in \FacesM, F \cap K \neq \emptyset}
h_F^{\frac{1}{2}}
\normLeb{\Jump{v}}{F}.\qedhere
\end{equation*}
\end{proof}

The local estimate in Proposition~\ref{P:stability-loc} ensures that $\Smt_\DG$ satisfies condition \eqref{EDG-conditions-stability}. 

\begin{theorem}[Stability of $\Smt_\DG$]
\label{T:stability}
The operator $\Smt_\DG$ defined in section~\ref{SS:construction-EDG} satisfies condition \eqref{EDG-conditions-stability} in that, for all $v \in (\Polyell{\degree})^2$, we have 
\begin{equation*}
\label{stability}
\normLeb{\Grad \Smt_\DG v}{\Domain}
\leq C \max\{ 1, 1/\sqrt{\eta} \} \norm{v}_\DG.
\end{equation*}
\end{theorem}

\begin{proof}
Let $K \in \Mesh$. An inverse estimate and Proposition~\ref{P:stability-loc} imply that
\begin{equation*}
\normLeb{\Grad(v-\Smt_\DG v)}{K}
\leq C 
\sum_{F \in \FacesM, F \cap K \neq \emptyset}
h_F^{-\frac{1}{2}}
\normLeb{\Jump{v}}{F}.
\end{equation*}
We square both sides in this inequality and sum over all $K \in \Mesh$. Recalling that the number of triangles touching a given edge is bounded in terms of the shape constant of $\Mesh$, we obtain
\begin{equation*}
\normLeb{\GradM (v - \Smt_\DG v)}{\Domain}
\leq C
\left( \int_\Skel h^{-1} \snorm{\Jump{v}}^2 \right)^{\frac{1}{2}}. 
\end{equation*}
We conclude by recalling the definition of the norm $\norm{\cdot}_\DG$ in section~\ref{SS:dG-discretization}. 
\end{proof}

\subsection{Main results}
\label{SS:main-results}

We are now able to derive the main result of this paper. For this purpose, we invoke \cite[Corollary~1]{Veeser:16} and derive the following upper bound of the velocity best error
\begin{equation*}
\inf_{w \in (\Polyell{\degree})^2}\norm{u-w}_\DG
\leq
\inf_{w \in (\Polyell{\degree} \cap \SobH{\Domain})^2}
\normLeb{\Grad(u-w)}{\Domain}
\leq C
\inf_{w \in (\Polyell{\degree})^2}
\normLeb{\GradM(u-w)}{\Domain}
\end{equation*}
for all $u \in \SobH{\Domain}^2$. Notice that the right-hand side is independent of the penalty parameter $\eta$ and bounds the left-hand side also from below. We combine this bound with Theorems~\ref{T:qopt-prerob}, \ref{T:conservation} and \ref{T:stability}. Recall also the definition of $\underline{\alpha}_\DG$ and the upper bounds of $\beta_\DG^{-1}$ and $\delta_\DG$ in \eqref{coercivity+boundedness+infsup-DG} and Lemma~\ref{L:discretely-divfree-approx}, respectively. 

\begin{theorem}[Quasi-optimality and pressure robustness by $\Smt_\DG$]
	\label{T:qopt-prerob-EDG}
	Let $\eta > \overline{\eta}$, where $\overline{\eta}$ is as in \eqref{DG-threshold}. Denote by $(u,p)$ and $(u_\DG, p_\DG)$ the solutions of \eqref{Stokes} and \eqref{Stokes-dG}, respectively, in dimension $\Dim = 2$, with viscosity $\Visc > 0$ and load $f \in \SobHD{\Domain}$. Moreover, let $\Smt_\DG$ be the operator defined in section~\ref{SS:construction-EDG}. Then, we have
	\begin{equation*}
	\norm{u-u_\DG}_\DG
	\leq C \sqrt{\eta} 
	\inf_{w \in (\Polyell{\degree})^2}
	\normLeb{\GradM(u-w)}{\Domain}
	\end{equation*}
	and
	\begin{equation*}
	\normLeb{p-p_\DG}{\Domain}
	\leq
	C \Visc \eta
	\inf_{w \in (\Polyell{\degree})^2}
	\normLeb{\GradM(u-w)}{\Domain}
	+
	\inf_{q \in \PolyellAvg{\degree-1}}
	\normLeb{p-q}{\Domain}. 
	\end{equation*}
\end{theorem} 

In section~\ref{SS:locking} we investigate numerically the impact of the penalty parameter $\eta$ on the error estimates, in connection with the discussion in section~\ref{SS:weak-penalization}. 

The above design of $\Smt_\DG$ can be simplified when the sole quasi-optimality (without pressure robustness) is concerned. In this case, we can apply the operator $\Smt$ from \cite[Proposition~3.4]{Veeser.Zanotti:18b} component-wise. This gives rise to
\begin{equation}
\label{EDG-simplified}
\widetilde \Smt_\DG v := (\Smt v_1, \Smt v_2), \qquad v = (v_1, v_2) \in (\Polyell{\degree})^2.
\end{equation}
According to \cite[Proposition~3.4]{Veeser.Zanotti:18b}, the resulting operator is moment-preserving and stable, in the sense that it satisfies conditions \eqref{EDG-conditions-stability}, \eqref{EDG-conditions-facemom} and \eqref{EDG-conditions-elemmom}. Then, the following weaker counterpart of Theorem~\ref{T:qopt-prerob-EDG} readily follows from Theorem~\ref{T:quasi-optimality}.

\begin{theorem}[Quasi-optimality by $\widetilde \Smt_\DG$]
	\label{T:quasi-optimality-EDG}
	Let $\eta > \overline{\eta}$, where $\overline{\eta}$ is as in \eqref{DG-threshold}. Denote by $(u,p)$ and $(u_\DG, p_\DG)$ the solutions of \eqref{Stokes} and \eqref{Stokes-dG}, respectively, in dimension $\Dim = 2$, with viscosity $\Visc > 0$ and load $f \in \SobHD{\Domain}$. Moreover, let $\Smt_\DG$ be replaced by $\widetilde{\Smt}_\DG$. Then, we have
	\begin{equation*}
	\Visc \norm{u-u_\DG}_\DG
	\leq C 
	\left( 
	\Visc \sqrt{\eta}
	\inf_{w \in (\Polyell{\degree})^2}
	\normLeb{\GradM(u-w)}{\Domain}
	+
	\inf_{q \in \PolyellAvg{\degree-1}}
	\normLeb{p-q}{\Domain}
	\right) 
	\end{equation*}
	and
	\begin{equation*}
	\normLeb{p-p_\DG}{\Domain}
	\leq C \sqrt{\eta}
	\left(  
	\Visc \sqrt{\eta}
	\inf_{w \in (\Polyell{\degree})^2}
	\normLeb{\GradM(u-w)}{\Domain}
	+
	\inf_{q \in \PolyellAvg{\degree-1}}
	\normLeb{p-q}{\Domain}
	\right) . 
	\end{equation*}
\end{theorem}

\subsection{Construction of $\Smt_\DG$ for \(d=3\)}
\label{sec:3d-smoother}

We end this section with some comments concerning the extension of the previous results to the discretization of the Stokes equations in dimension $\Dim = 3$. First of all, the three-dimensional $\Curl$ operator has to be used instead of the two-dimensional operators $\Curl$ and $\Rot$ from \eqref{curl-rot-d=2}. The decomposition of vector-valued polynomials stated in Lemma~\ref{L:decompositions} and used in the definition of $\Smt_4$ reads    
\begin{equation*}
\Poly{k}(K)^3
=
\Grad \Poly{k+1}(K)
\oplus
x\wedge \Poly{k-1}(K)^3
\end{equation*}
for all $K \in \Mesh$ and $k \geq 0$, where
\begin{equation*}
x \wedge \Poly{k-1}(K)^3:=
\left \{ x \wedge r :=
\left( 
\begin{tabular}{c}
$x_2 r_3 - x_3 r_2 $\\
$x_3 r_1 - x_1 r_3 $\\
$x_1 r_2 - x_2 r_1 $
\end{tabular}\right) \mid r = (r_1, r_2, r_3) \in \Poly{k-1}(K)^3
\right\}.
\end{equation*}
This can be verified by noticing that the operator $\Curl$ is injective on $x \wedge \Poly{k-1}(K)^3$.

The construction of $\Smt_\DG$ remains the same as in section~\ref{SS:construction-EDG}, up to the following minor modifications. The face bubble function $b_F$ involved in the definition of $\Smt_2$ has degree three (and not two). Consequently, the operator $\Smt_{3, \Ref}$ maps $\PolyellAvg{\degree+1}(\Mesh_\Ref)$ into $\Lagr{\degree+2}(\Mesh_\Ref)^3$, so as to guarantee that $\Smt_3$ is well-defined. Finally, the volume bubble function $b_\Ref$ involved in the definition of $\Smt_4$ has degree four (and not three). 

With these ingredients, the statements and the proofs of the results in sections~\ref{SS:preservation-EDG}-\ref{SS:main-results} can be easily adapted to the case $\Dim = 3$.

\section{Numerical experiments}
\label{S:numerical-experiments}

We now discuss the results obtained when approximating the solution of the Stokes equations \eqref{Stokes} with
\begin{equation*}
\Dim = 2 \qquad \qquad
\Domain = (0,1) \times (0,1) \qquad \qquad
\Visc = 1.
\end{equation*}

We discretize the domain $\Domain$ by the following two families of meshes. Given $N \geq 0$, we divide $\Domain$ into $2^N \times 2^N$ identical squares with area $2^{-2N}$. Then, we obtain the `diagonal' mesh $\Mesh^D_N$ by drawing the diagonal with positive slope of each square. Similarly, we obtain the `crisscross' mesh $\Mesh^C_N$ by drawing both diagonals of each square, see Figure~\ref{F:meshes}. 

We test the discretization \eqref{Stokes-dG} of the Stokes equations with
\begin{itemize}
	\item $\Smt_\DG = \mathrm{Id}$ as in \cite{Hansbo.Larson:02} (standard discretization),
	\item $\Smt_\DG$ defined by \eqref{EDG-simplified} as in Theorem~\ref{T:quasi-optimality-EDG} (quasi-optimal discretization),
	\item $\Smt_\DG$ defined by section~\ref{SS:construction-EDG} as in Theorem~\ref{T:qopt-prerob-EDG} (quasi-optimal and pressure robust discretization).
\end{itemize}
The first option differs from the others, in that $\Smt_\DG$ does not map $(\Polyell{\degree})^2$ into $\SobH{\Domain}^2$. Therefore, the duality in the right-hand side of \eqref{Stokes-dG} is not defined for a general load $f \in \SobHD{\Domain}$. This observation clearly favors the second and the third discretizations when rough loads are concerned, cf. \cite[section~6.4]{Verfuerth.Zanotti:19}.

We consider only the first-order discretization in our experiments, i.e. we set 
\begin{equation*}
\degree = 1
\end{equation*}
in \eqref{Stokes-dG}. The numerical results are obtained with the help of ALBERTA 3.0 \cite{Heine.Koester.Kriessl.Schmidt.Siebert,Schmidt.Siebert:05}.

\begin{figure}[htp]
	\hfill
	\subfloat{\includegraphics[width=0.4\hsize]{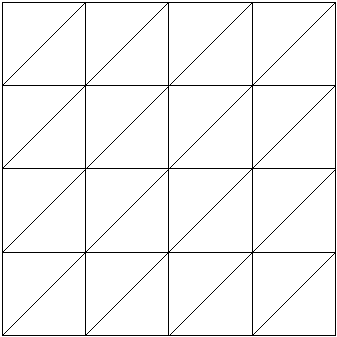}}
	\hfill
	\subfloat{\includegraphics[width=0.4\hsize]{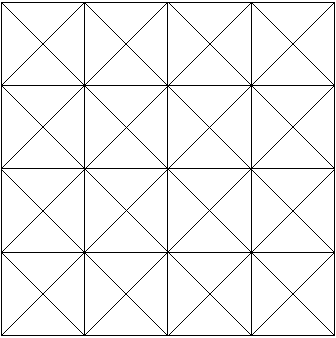}}
	\hfill
	\caption{Diagonal mesh $\Mesh_N^D$ (left) and crisscross mesh $\Mesh_N^C$ (right) with $N=2$.}
	\label{F:meshes}
\end{figure}

\subsection{Smooth exact solution}
\label{SS:smooth-exact-solution}

We first consider a test case with smooth exact solution, namely
\begin{equation*}
u(x_1, x_2) = \Curl(x_1^2(1-x_1)^2x_2^2(1-x_2)^2 )
\qquad
p(x_1, x_2) = (x_1 - 0.5)(x_2 - 0.5).
\end{equation*}
We use the crisscross meshes $\Mesh^C_N$ with $N \in \{ 0, 1, \dots 8 \}$ and the penalty parameter
\begin{equation*}
\eta = 6.
\end{equation*}
We report some values of the velocity error $\norm{u-u_\DG}_\DG$ and of the pressure error $\normLeb{p-p_\DG}{\Domain}$, for the three discretizations listed above, in Tables~\ref{F:smooth-sol-velocity} and \ref{F:smooth-sol-pressure}, respectively. For each sequence of errors $(e_N)$, we compute the experimental order of convergence
\begin{equation*}
\label{EOC}
\mathrm{EOC}_N 
:= 
\frac{\log(e_N / e_{N-1})}{\log(\#\Mesh_{N-1}^C / \#\Mesh_N^C)}, \qquad N \geq 1 
\end{equation*}
where $\#\Mesh_N^C$ denotes the number of triangles in $\Mesh_N^C$. Observing the numerical data, we see that the errors of the three discretizations behave quite similarly and converge to zero at the maximum decay rate $(\#\Mesh_N^C)^{-0.5}$. In this case, the standard discretization should be preferred for the easier construction of the operator $\Smt_\DG$.  

\begin{table}[htp]
	\begin{tabular}{r|cc|cc|cc}
		N &
		$\mathtt{stnd}$  & EOC &
		$\mathtt{qopt}$  & EOC &
		$\mathtt{prob}$  & EOC \\[1ex]
		\hline &&\\[-1.5ex]
		4 &
		8.2516e-03 &
		&      
		8.3795e-03 &
		&    
		8.5337e-03 &  
		\\
		5 &
		3.8937e-03 &  
		\raisebox{1.5ex}[0pt]{0.54} &
		3.9344e-03 &  
		\raisebox{1.5ex}[0pt]{0.55} & 
		4.1273e-03 &   
		\raisebox{1.5ex}[0pt]{0.52}
		\\
		6 &  
		1.8797e-03 &  
		\raisebox{1.5ex}[0pt]{0.53} &    
		1.8910e-03 &  
		\raisebox{1.5ex}[0pt]{0.53} &   
		2.0231e-03 &  
		\raisebox{1.5ex}[0pt]{0.51}
		\\
		7 &  
		9.2180e-04 &
		\raisebox{1.5ex}[0pt]{0.51} &    
		9.2477e-04 &  
		\raisebox{1.5ex}[0pt]{0.52} &   
		1.0007e-03 &  
		\raisebox{1.5ex}[0pt]{0.51}
		\\
		8 &  
		4.5621e-04 &  
		\raisebox{1.5ex}[0pt]{0.51} &   
		4.5698e-04 &  
		\raisebox{1.5ex}[0pt]{0.51} &   
		4.9756e-04 &  
		\raisebox{1.5ex}[0pt]{0.50}
	\end{tabular}
	\caption{Section~\ref{SS:smooth-exact-solution}. Velocity errors of the standard ($\mathtt{stnd}$), quasi-optimal ($\mathtt{qopt}$) and quasi-optimal and pressure robust ($\mathtt{prob}$) discretizations
		with experimental orders of convergence.}
\label{F:smooth-sol-velocity}
\end{table}

\begin{table}[htp]
	\begin{tabular}{r|cc|cc|cc}
		N &
		$\mathtt{stnd}$  & EOC &
		$\mathtt{qopt}$  & EOC &
		$\mathtt{prob}$  & EOC \\[1ex]
		\hline &&\\[-1.5ex]
		 4 &  
		 4.4477e-03 &
		 & 
		 4.4862e-03 &  
		 &    
		 4.3843e-03 &  
		 \\
		 5 &  
		 2.2248e-03 &   
		 \raisebox{1.5ex}[0pt]{0.50} &    
		 2.2377e-03 &  
		 \raisebox{1.5ex}[0pt]{0.50} &   
		 2.2109e-03 &  
		 \raisebox{1.5ex}[0pt]{0.49}
		 \\
		 6 &  
		 1.1142e-03 &  
		 \raisebox{1.5ex}[0pt]{0.50} &   
		 1.1178e-03 &  
		 \raisebox{1.5ex}[0pt]{0.50} &   
		 1.1109e-03 &   
		 \raisebox{1.5ex}[0pt]{0.50}
		 \\
		 7 &  
		 5.5781e-04 &  
		 \raisebox{1.5ex}[0pt]{0.50} &   
		 5.5878e-04 &  
		 \raisebox{1.5ex}[0pt]{0.50} &   
		 5.5692e-04 &  
		 \raisebox{1.5ex}[0pt]{0.50}
		 \\
		 8 &  
		 2.7912e-04 &  
		 \raisebox{1.5ex}[0pt]{0.50} &   
		 2.7937e-04 &  
		 \raisebox{1.5ex}[0pt]{0.50} &   
		 2.7884e-04 &  
		 \raisebox{1.5ex}[0pt]{0.50}
	\end{tabular}
	\caption{Section~\ref{SS:smooth-exact-solution}. Pressure errors of the standard ($\mathtt{stnd}$), quasi-optimal ($\mathtt{qopt}$) and quasi-optimal and pressure robust ($\mathtt{prob}$) discretizations with experimental orders of convergence.}
	\label{F:smooth-sol-pressure}
\end{table}

\subsection{Jumping pressure}
\label{SS:jumping-pressure}

In order to investigate the pressure robustness of the three discretizations, we consider a test case with smooth exact velocity and rough exact pressure, namely
\begin{equation*}
u(x_1, x_2) = \Curl\left(\:x_1^2(x_1-1)^2x_2^2(x_2-1)^2\:\right),
\quad
p(x_1, x_2) = \begin{cases}
\frac{\pi}{\pi-1} & \text{if} \;\; x_1 > \pi^{-1} \\
-\pi & \text{if} \;\; x_1 < \pi^{-1}
\end{cases}.
\end{equation*}
As before, we use the crisscross meshes $\Mesh^C_N$ with $N \in \{ 0, 1, \dots, 8 \}$ and the penalty parameter $\eta = 6$. Note that the meshes do not resolve the discontinuity of $p$ along the line $x_1 = \pi^{-1}$.

The data displayed in Figure~\ref{F:PreDisc} show that the velocity error of the quasi-optimal and pressure robust discretization fully exploits the regularity of $u$ and converges to zero at the maximum decay rate $(\#\Mesh^C_N)^{-0.5}$, in accordance with Theorem~\ref{T:qopt-prerob-EDG}. In contrast, the low regularity of $p$ impairs the approximation of $u$ in the standard discretization and in the quasi-optimal one. In fact, the corresponding velocity errors converge at the suboptimal decay rate $(\#\Mesh^C_N)^{-0.25}$. This confirms, in particular, that the first estimate in Theorem~\ref{T:quasi-optimality-EDG} captures the correct behavior of the velocity error in the quasi-optimal discretization.  

\begin{figure}[htp]
	\includegraphics[width=0.4\hsize]{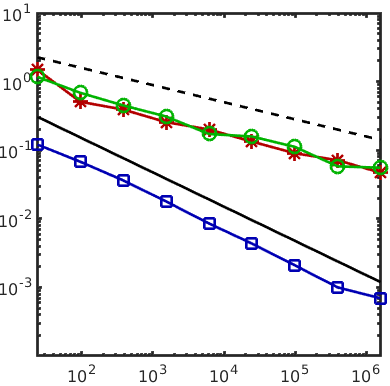}
	\caption{Section~\ref{SS:jumping-pressure}. Velocity error as a function of $\#\Mesh^C_N$ for the standard ($\circ$), quasi-optimal ($*$) and quasi-optimal and pressure robust ($\square$) discretizations. Plain and dashed lines indicate the decay rates $(\#\Mesh^C_N)^{-0.5}$ and $(\#\Mesh^C_N)^{-0.25}$, respectively.}
	\label{F:PreDisc}
\end{figure}

\subsection{Locking}
\label{SS:locking}

Finally, we investigate the robustness of the three discretizations with respect to the penalty parameter $\eta$. To this end, we consider the same exact solution as in section~\ref{SS:smooth-exact-solution}, with
\begin{equation*}
\eta \in \{ 10, 100, 1000 \}.
\end{equation*} 
For a fair comparison, we measure the velocity errors in the parameter-independent norm
\begin{equation*}
\norm{\cdot}^2_{\DG,1} := 
\normLeb{\GradM \cdot}{\Domain}^2
+
\int_\Skel \dfrac{1}{h} \snorm{\Jump{\cdot}}^2.
\end{equation*}
Note that $\norm{\cdot}_{\DG,1} \leq \norm{\cdot}_\DG$ for the considered values of $\eta$. Since the three discretizations produce qualitatively similar results, we pick the quasi-optimal and pressure robust discretization as representative of the others.

We discretize the domain by the diagonal meshes $\Mesh^D_N$ with $N \in \{ 0, 1, \dots, 7 \}$. This choice is motivated by the fact that the constant $\delta_\mathrm{SV}$ from \eqref{delta-SV} is proportional to $(\Mesh_N^D)^{0.5}$ as a consequence of $Z_\mathrm{SV} = \{0\}$, see  \cite[equation~11.3.8]{Brenner.Scott:08}. Hence, according to the discussion in section~\ref{SS:weak-penalization}, we expect to observe locking. The results displayed in Figure~\ref{F:locking} confirm our expectation. In particular, we observe that the pressure error is more sensitive to the size of $\eta$ than the velocity error, in accordance with the estimates in Theorem~\ref{T:qopt-prerob-EDG}.

One way to achieve robustness consists in weakening the jump penalization in the form $a_\DG$, as suggested in \eqref{weak-penalization}. With this modification, the results are almost insensitive to the size of $\eta$. Still, it has to be said that we obtain larger velocity errors than before for moderate values of $\eta$. 

\begin{figure}[htp]
	\hfill
	\subfloat{\includegraphics[width=0.48\hsize]{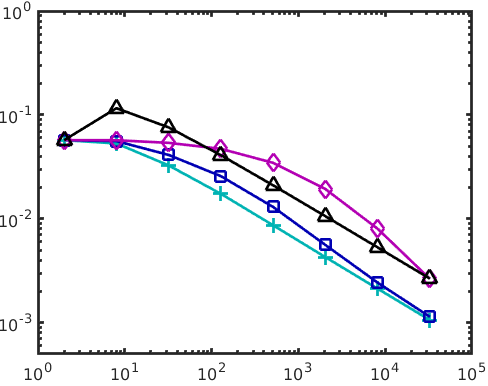}}
	\hfill
	\subfloat{\includegraphics[width=0.48\hsize]{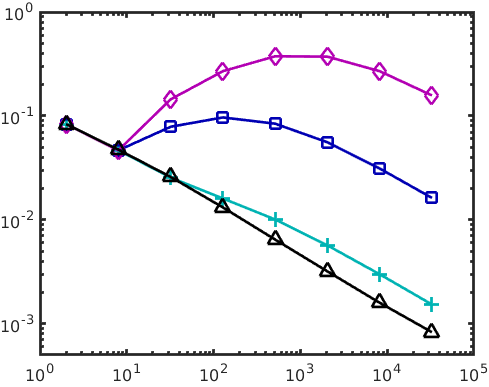}}
	\hfill
	\caption{Section~\ref{SS:locking}. Velocity (left) and pressure (right) errors of the quasi-optimal and pressure robust discretization as functions of $\#\Mesh_N^D$ for $\eta = 10$ ($+$), $\eta = 100$ ($\square$) and $\eta = 1000$ ($\lozenge$). The variant with weak jump penalization is also considered ($\triangle$) and the results for the three values of $\eta$ graphically coincide.} 
	\label{F:locking}
\end{figure}

\subsection*{Funding}
Christian Kreuzer gratefully acknowledges partial support by the DFG research grant “Convergence Analysis for Adaptive Discontinuous Galerkin Methods” (KR 3984/5-1). The research of Pietro Zanotti was supported by the GNCS-INdAM through
the program ``Finanziamento giovani ricercatori 2019-2020''.

%--BIBLIOGRAPHY--------------------------------------------------

\end{document}